\documentclass[12pt]{amsart}
\usepackage{amsmath,amsthm,amsfonts,amssymb,mathrsfs}
\date{\today}

\usepackage[cp1251]{inputenc}         

\usepackage[ukrainian]{babel} 
\usepackage{amsmath,amsthm,amsfonts,amssymb,mathrsfs}
\usepackage{eucal}

\input xymatrix
\xyoption{all}

\usepackage{color}

\usepackage{amssymb,cite}
\usepackage[colorlinks,plainpages,citecolor=magenta, linkcolor=blue, backref]{hyperref}

\usepackage{hyperref}

\newtheorem{theorem}{Теорема}

\newtheorem{proposition}{Твердження}
\newtheorem{corollary}{Наслiдок}
\newtheorem{lemma}{Лема}
\theoremstyle{definition}
\newtheorem{example}{Приклад}
\newtheorem{remark}{Зауваження}

\newtheorem{definition}[theorem]{Означення}

\begin{document}

\title[Про груповi конгруенцi\"{\i} на напiвгрупi $\boldsymbol{B}_{\omega}^{\mathscr{F}}$ та \"{\i}\"{\i} гомоморфнi ретракти ...]{Про груповi конгруенцi\"{\i} на напiвгрупi $\boldsymbol{B}_{\omega}^{\mathscr{F}}$ та \"{\i}\"{\i} гомоморфнi ретракти у випадку, коли сiм'я $\mathscr{F}$ складається з непорожнiх iндуктивних пiдмножин у  ${\omega}$}

\author[Олег~Гутік, Микола Михаленич]{Олег~Гутік, Микола Михаленич}
\address{Механіко-математичний факультет, Львівський національний університет ім. Івана Франка, Університецька 1, Львів, 79000, Україна}
\email{oleg.gutik@lnu.edu.ua,
ovgutik@yahoo.com, myhalenychmc@gmail.com}



\begin{abstract}
Ми вивчаємо групові конгруенції на напівгрупі $\boldsymbol{B}_{\omega}^{\mathscr{F}}$ та її гомоморфні ретракти у випадку, коли $\mathscr{F}$ --- ${\omega}$-замкнена сім'я  з індуктивних непорожніх підмножин в $\omega$. Доведено, що конгруенція $\mathfrak{C}$ на $\boldsymbol{B}_{\omega}^{\mathscr{F}}$ є груповою, тоді і лише тоді, коли звуження конгруенції $\mathfrak{C}$ на піднапівгрупу в $\boldsymbol{B}_{\omega}^{\mathscr{F}}$, яка ізоморфна біциклічній напівгруп, не є відношенням рівності. Також, ми описуємо всі нетривіальні гомоморфні ретракти та ізоморфізми напівгрупи $\boldsymbol{B}_{\omega}^{\mathscr{F}}$.

\bigskip
\noindent
\emph{Oleg Gutik, Mykola Mykhalenych, \textbf{On group congruences on the semigroup $\boldsymbol{B}_{\omega}^{\mathscr{F}}$ and its homo\-mor\-phic retracts in the case when a family $\mathscr{F}$ consists of inductive non-empty subsets of~$\omega$}.}

\smallskip
\noindent
We study group congruences on the semigroup  $\boldsymbol{B}_{\omega}^{\mathscr{F}}$ and its homomorphic retracts  in the case when an ${\omega}$-closed family $\mathscr{F}$ which consists of inductive non-empty subsets of $\omega$. It is proven that a congruence $\mathfrak{C}$ on $\boldsymbol{B}_{\omega}^{\mathscr{F}}$ is a group congruence if and only if its restriction on a subsemigroup of $\boldsymbol{B}_{\omega}^{\mathscr{F}}$, which is isomorphic to the bicyclic semigroup, is not the identity relation. Also, all non-trivial homomorphic retracts and isomorphisms  of the semigroup $\boldsymbol{B}_{\omega}^{\mathscr{F}}$ are described.
\end{abstract}

\maketitle


\section{Вступ}\label{section-1}

У цій праці ми користуємося термінологією з монографій \cite{Clifford-Preston-1961, Clifford-Preston-1967, Lawson-1998, Petrich-1984}.
Надалі у тексті множину невід'ємних цілих чисел  позначатимемо через $\omega$. Для довільного числа $k\in\omega$ позначимо
${[k)=\{i\in\omega\colon i\geqslant k\}.}$

Нехай $\mathscr{P}(\omega)$~--- сім'я усіх підмножин у $\omega$.
Для довільних $F\in\mathscr{P}(\omega)$ та $n,m\in\omega$ покладемо $n-m+F=\{n-m+k\colon k\in F\}$.
Будемо говорити, що непорожня підсім'я  $\mathscr{F}\subseteq\mathscr{P}(\omega)$ є \emph{${\omega}$-замкненою}, якщо $F_1\cap(-n+F_2)\in\mathscr{F}$ для довільних $n\in\omega$ та $F_1,F_2\in\mathscr{F}$.

Підмножина $A$ в $\omega$ називається \emph{індуктивною}, якщо з  $i\in A$ випливає, що $i+1\in A$. Очевидно, що $\varnothing$ --- індуктивна множина в $\omega$.

\begin{remark}\label{remark-2.2}
\begin{enumerate}
  \item\label{remark-2.2(1)} За лемою 6 з \cite{Gutik-Mykhalenych-2020} непорожня множина $F\subseteq \omega$ є індуктивною в $\omega$ тоді і лише тоді, коли $(-1+F)\cap F=F$.
  \item\label{remark-2.2(2)} Оскільки множина $\omega$ зі звичайним порядком є цілком впорядкованою, то для кожної непорожньої індуктивної множини $F$ у $\omega$ існує невід'ємне ціле число $n_F\in\omega$ таке, що $[n_F)=F$.
  \item\label{remark-2.2(3)} З \eqref{remark-2.2(2)} випливає, що перетин довільної скінченної кількості непорожніх індуктивних підмножин у $\omega$ є непорожньою індуктивною підмножиною в $\omega$. 
\end{enumerate}
\end{remark}

Якщо $S$~--- напівгрупа, то її підмножина ідемпотентів позначається через $E(S)$.  На\-пів\-гру\-па $S$ називається \emph{інверсною}, якщо для довільного її елемента $x$ існує єдиний елемент $x^{-1}\in S$ такий, що $xx^{-1}x=x$ та $x^{-1}xx^{-1}=x^{-1}$ \cite{Petrich-1984, Vagner-1952}. В інверсній напівгрупі $S$ вище означений елемент $x^{-1}$ називається \emph{інверсним до} $x$. \emph{В'язка}~--- це напівгрупа ідемпотентів, а \emph{напівґратка}~--- це комутативна в'язка. 


Відношення еквівалентності $\mathfrak{K}$ на напівгрупі $S$ називається \emph{конгруенцією}, якщо для елементів $a$ та $b$ напівгрупи $S$ з того, що виконується умова $(a,b)\in\mathfrak{K}$ випливає, що $(ca,cb), (ad,bd) \in\mathfrak{K}$, для довільних $c,d\in S$. Відношення $(a,b)\in\mathfrak{K}$ ми також будемо записувати $a\mathfrak{K}b$, і в цьому випадку будемо говорити, що \emph{елементи $a$ i $b$ є $\mathfrak{K}$-еквівалентними}.

Конгруенція $\mathfrak{K}$ на напівгрупі $S$ називається \emph{груповою}, якщо фактор-напівгрупа $S/\mathfrak{K}$ ізоморфна деякій групі $G$ \cite{Clifford-Preston-1961}. Нагадаємо \cite{Lawson-1998, Petrich-1984}, що на кожній ін\-версній напівгрупі $S$ існує \emph{найменша} (\emph{мінімальна}) \emph{групова конгруенція} $\boldsymbol{\sigma}$ і вона визначається так:
\begin{equation*}
  s\boldsymbol{\sigma}t \qquad \Longleftrightarrow \qquad es=et \quad \hbox{для деякого} \quad e\in E(S).
\end{equation*}

Якщо $S$~--- напівгрупа, то на $E(S)$ визначено частковий порядок:
$
e\preccurlyeq f
$   тоді і лише тоді, коли
$ef=fe=e$.
Так означений частковий порядок на $E(S)$ називається \emph{при\-род\-ним}.

Означимо відношення $\preccurlyeq$ на інверсній напівгрупі $S$ так:
$
    s\preccurlyeq t
$
тоді і лише тоді, коли $s=te$, для деякого ідемпотента $e\in S$. Так означений частковий порядок назива\-єть\-ся \emph{при\-род\-ним част\-ковим порядком} на інверсній напівгрупі $S$~\cite{Vagner-1952}. Очевидно, що звуження природного часткового порядку $\preccurlyeq$ на інверсній напівгрупі $S$ на її в'язку $E(S)$ є при\-род\-ним частковим порядком на $E(S)$.

Нагадаємо (див.  \cite[\S1.12]{Clifford-Preston-1961}, що \emph{біциклічною напівгрупою} (або \emph{біциклічним моноїдом}) ${\mathscr{C}}(p,q)$ називається напівгрупа з одиницею, породжена двоелементною мно\-жи\-ною $\{p,q\}$ і визначена одним  співвідношенням $pq=1$. Біциклічна на\-пів\-група відіграє важливу роль у теорії
на\-півгруп. Так, зокрема, класична теорема О.~Ан\-дерсена \cite{Andersen-1952}  стверджує, що {($0$-)}прос\-та напівгрупа з (ненульовим) ідем\-по\-тен\-том є цілком {($0$-)}прос\-тою тоді і лише тоді, коли вона не містить ізоморфну копію бі\-циклічного моноїда. Різні розширення та узагальнення біциклічного моноїда вводилися раніше різ\-ни\-ми авторами \cite{Fortunatov-1976, Fotedar-1974, Fotedar-1978, Gutik-Pagon-Pavlyk=2011, Warne-1967}. Такими, зокрема, є конструкції Брука та Брука--Рейлі занурення напівгруп у прості та описання інверсних біпростих і $0$-біпростих $\omega$-напівгруп \cite{Bruck-1958, Reilly-1966, Warne-1966, Gutik-2018}.

\emph{Гомоморфною ретракцією} називається відображення з напівгрупи $S$ в $S$, яке є одночасно рет\-рак\-цією та гомоморфізмом \cite{Clifford-Preston-1961, Clifford-Preston-1967}. Образ напівгрупи $S$ при її гомоморфній ретракції називається \emph{гомоморфним рет\-рак\-том}. Тобто гомоморфний ретракт напівгрупи $S$~--- це така піднапівгрупа $T$ в $S$, що існує гомоморфізм з $S$ на $T$, для якого піднапівгрупа $T$ є множиною всіх його нерухомих точок. Терміни ``гомоморфні ретракції'' та ``гомоморфні ретракти'', здається, вперше з’явились у праці Брауна \cite{Brown-1965} при дослідженні структури топологічних напівґраток. Очевидно, що кожне тотожне відображення напівгрупи $S$ є її гомоморфною ретракцією, а також, якщо $e$~--- ідемпотент в $S$, то стале відображення $h\colon S\to S$, $x\mapsto e$ є гомоморфною рет\-рак\-цією напівгрупи $S$. Такі гомоморфні ретракції та тотожне відображення напівгрупи $S$ будемо називати \emph{тривіальними}, а  образи напівгрупи $S$ стосовно них --- \emph{тривіальними} гомоморфними ретрактами.

\begin{remark}\label{remark-10}
Легко бачити, що біциклічний моноїд ${\mathscr{C}}(p,q)$ ізоморфний напівгрупі, заданій на множині $\boldsymbol{B}_{\omega}=\omega\times\omega$ з напівгруповою операцією
\begin{equation*}
  (i_1,j_1)\cdot(i_2,j_2)=(i_1+i_2-\min\{j_1,i_2\},j_1+j_2-\min\{j_1,i_2\})=
\left\{
  \begin{array}{ll}
    (i_1-j_1+i_2,j_2), & \hbox{якщо~} j_1\leqslant i_2;\\
    (i_1,j_1-i_2+j_2), & \hbox{якщо~} j_1\geqslant i_2.
  \end{array}
\right.
\end{equation*}
\end{remark}

У праці \cite{Gutik-Mykhalenych-2020} введено алгебраїчні розширення $\boldsymbol{B}_{\omega}^{\mathscr{F}}$ біциклічного моноїда для довільної $\omega$-замк\-не\-ної сім'ї $\mathscr{F}$ підмножин в $\omega$, які узагальнюють біциклічний моноїд, зліченну напівгрупу матричних одиниць і деякі інші комбінаторні інверсні напівгрупи. 

Нагадаємо цю конструкцію.
Нехай $\boldsymbol{B}_{\omega}$~--- біциклічний моноїд і  $\mathscr{F}$ --- непорожня ${\omega}$-замкнена підсім'я в  $\mathscr{P}(\omega)$. На множині $\boldsymbol{B}_{\omega}\times\mathscr{F}$ озна\-чимо бінарну операцію ``$\cdot$''  формулою
\begin{equation*}
  (i_1,j_1,F_1)\cdot(i_2,j_2,F_2)=
  \left\{
    \begin{array}{ll}
      (i_1-j_1+i_2,j_2,(j_1-i_2+F_1)\cap F_2), & \hbox{якщо~} j_1<i_2;\\
      (i_1,j_2,F_1\cap F_2),                   & \hbox{якщо~} j_1=i_2;\\
      (i_1,j_1-i_2+j_2,F_1\cap (i_2-j_1+F_2)), & \hbox{якщо~} j_1>i_2.
    \end{array}
  \right.
\end{equation*}

У \cite{Gutik-Mykhalenych-2020} доведено, якщо сім'я  $\mathscr{F}\subseteq\mathscr{P}(\omega)$ є ${\omega}$-замкненою, то $(\boldsymbol{B}_{\omega}\times\mathscr{F},\cdot)$ є напівгрупою.

У \cite{Gutik-Mykhalenych-2020} доведено, що $\boldsymbol{B}_{\omega}^{\mathscr{F}}$ є комбінаторною інверсною напівгрупою, а також описано відношення Ґріна, частковий природний порядок на напівгрупі $\boldsymbol{B}_{\omega}^{\mathscr{F}}$ та її множину ідемпотентів. Також, у \cite{Gutik-Mykhalenych-2020} доведено критерії  простоти, $0$-простоти, біпростоти та $0$-біпростоти напівгрупи $\boldsymbol{B}_{\omega}^{\mathscr{F}}$, і вказано умови, коли $\boldsymbol{B}_{\omega}^{\mathscr{F}}$ містить одиницю, ізоморфна біциклічному моноїду або зліченній напівгрупі матричних одиниць.

Зауважимо, що у \cite{Gutik-Pozdnyakova-2021??} отримано подібні результати до \cite{Gutik-Mykhalenych-2020} у випадку розширення $\boldsymbol{B}_{\mathbb{Z}}^{\mathscr{F}}$ розширеної біциклічної напівгрупи $\boldsymbol{B}_{\mathbb{Z}}$ для довільної $\omega$-замк\-не\-ної сім'ї $\mathscr{F}$ підмножин в $\omega$.

Припустимо, що ${\omega}$-замкнена сім'я $\mathscr{F}\subseteq\mathscr{P}(\omega)$ містить порожню множину $\varnothing$, то з означення напівгрупової операції $(\boldsymbol{B}_{\omega}\times\mathscr{F},\cdot)$ випливає, що множина
$ 
  \boldsymbol{I}=\{(i,j,\varnothing)\colon i,j\in\omega\}
$ 
є ідеалом напівгрупи $(\boldsymbol{B}_{\omega}\times\mathscr{F},\cdot)$.

\begin{definition}[\!\!{\cite{Gutik-Mykhalenych-2020}}]\label{definition-1.1}
Для довільної ${\omega}$-замкненої сім'ї $\mathscr{F}\subseteq\mathscr{P}(\omega)$ означимо
\begin{equation*}
  \boldsymbol{B}_{\omega}^{\mathscr{F}}=
\left\{
  \begin{array}{ll}
    (\boldsymbol{B}_{\omega}\times\mathscr{F},\cdot)/\boldsymbol{I}, & \hbox{якщо~} \varnothing\in\mathscr{F};\\
    (\boldsymbol{B}_{\omega}\times\mathscr{F},\cdot), & \hbox{якщо~} \varnothing\notin\mathscr{F}.
  \end{array}
\right.
\end{equation*}
\end{definition}

У \cite{Gutik-Lysetska=2021, Lysetska=2020} досліджено алгебраїчну структуру напівгрупи $\boldsymbol{B}_{\omega}^{\mathscr{F}}$ у випадку, коли ${\omega}$-замкнена сім'я $\mathscr{F}$ складається з атомарних підмножин (одноточкових підмножин і порожньої множини) в ${\omega}$. Зокрема доведено, що за виконання таких умов на сім'ю $\mathscr{F}$ напівгрупа  $\boldsymbol{B}_{\omega}^{\mathscr{F}}$ ізоморфна піднапівгрупі ${\omega}$-розширення Брандта напівгратки $(\omega,\min)$. Також, у \cite{Gutik-Lysetska=2021, Lysetska=2020} досліджувались топологізація напівгрупи $\boldsymbol{B}_{\omega}^{\mathscr{F}}$, близькі до компактних трансляційно неперервні топології на $\boldsymbol{B}_{\omega}^{\mathscr{F}}$ та замикання напівгрупи $\boldsymbol{B}_{\omega}^{\mathscr{F}}$ у напівтопологічних напівгрупах.

Із зауваження~\ref{remark-2.2}\eqref{remark-2.2(3)} випливає, якщо сім'я $\mathscr{F}_0$ склада\-єть\-ся з індуктивних в $\omega$ підмножин і міс\-тить порожню множину $\varnothing$ як елемент, то для сім'ї $\mathscr{F}=\mathscr{F}_0\setminus\{\varnothing\}$ множина  $\boldsymbol{B}_{\omega}^{\mathscr{F}}$ з індукованою напівгруповою операцією з  $\boldsymbol{B}_{\omega}^{\mathscr{F}_0}$ є піднапівгрупою в $\boldsymbol{B}_{\omega}^{\mathscr{F}_0}$.

У цій праці ми вивчаємо алгебраїчну структуру напівгрупи $\boldsymbol{B}_{\omega}^{\mathscr{F}}$ у випадку, коли ${\omega}$-замкнена сім'я $\mathscr{F}$ складається з індуктивних непорожніх підмножин у $\omega$, а саме групові конгруенції на напівгрупі $\boldsymbol{B}_{\omega}^{\mathscr{F}}$ та її гомоморфні ретракти. Доведено, що конгруенція $\mathfrak{C}$ на $\boldsymbol{B}_{\omega}^{\mathscr{F}}$ є груповою, тоді і лише тоді, коли звуження конгруенції $\mathfrak{C}$ на піднапівгрупу в $\boldsymbol{B}_{\omega}^{\mathscr{F}}$, яка ізоморфна біциклічній напівгрупі, не є відношенням рівності. Також, ми описуємо всі нетривіальні гомоморфні ретракти та ізоморфізми напівгрупи $\boldsymbol{B}_{\omega}^{\mathscr{F}}$.

Надалі скрізь в тексті ми вважаємо, що ${\omega}$-замкнена сім'я $\mathscr{F}$ складається лише з індуктивних непорожніх підмножин у $\omega$.
\section{Алгебраїчні властивості напівгрупи $\boldsymbol{B}_{\omega}^{\mathscr{F}}$}\label{section-2}

\begin{proposition}\label{proposition-2.1}
Нехай $\mathscr{F}$ --- довільна ${\omega}$-замкнена сім'я підмножин у $\omega$ і $n_0=\min \displaystyle\big\{\bigcup\mathscr{F}\big\}$. Тоді:
\begin{enumerate}
  \item\label{proposition-2.1(1)} $\mathscr{F}_0=\left\{-n_0+F\colon F\in\mathscr{F}\right\}$ --- ${\omega}$-замкнена сім'я підмножин у $\omega$;
  \item\label{proposition-2.1(2)} напівгрупи $\boldsymbol{B}_{\omega}^{\mathscr{F}}$ і $\boldsymbol{B}_{\omega}^{\mathscr{F}_0}$ ізоморфні.
\end{enumerate}
\end{proposition}

\begin{proof}
Спочатку зауважимо, оскільки множина $\omega$ зі звичайним порядком $\leqslant$ є цілком впорядкованою, то невід'ємне ціле число $n_0$ визначено коректно.

Твердження \eqref{proposition-2.1(1)} очевидне.

\eqref{proposition-2.1(2)} Очевидно, що відображення $\mathfrak{h}\colon (\boldsymbol{B}_{\omega}\times\mathscr{F},\cdot)\to(\boldsymbol{B}_{\omega}\times\mathscr{F}_0,\cdot)$, означене за формулою
\begin{equation*}
  \mathfrak{h}(i,j,F)=(i,j,-n_0+F),
\end{equation*}
бієктивне. Врахувавши, що $-n_0+\varnothing=\varnothing$,
\begin{align*}
   \mathfrak{h}((i_1,j_1,F_1)\cdot(i_2,j_2,F_2))&=
  \left\{
    \begin{array}{ll}
       \mathfrak{h}(i_1-j_1+i_2,j_2,(j_1-i_2+F_1)\cap F_2), & \hbox{якщо~} j_1<i_2;\\
       \mathfrak{h}(i_1,j_2,F_1\cap F_2),                   & \hbox{якщо~} j_1=i_2;\\
       \mathfrak{h}(i_1,j_1-i_2+j_2,F_1\cap (i_2-j_1+F_2)), & \hbox{якщо~} j_1>i_2
    \end{array}
  \right.=\\
  &=
  \left\{
    \begin{array}{ll}
      (i_1-j_1+i_2,j_2,-n_0+((j_1-i_2+F_1)\cap F_2)), & \hbox{якщо~} j_1<i_2;\\
      (i_1,j_2,-n_0+(F_1\cap F_2)),                   & \hbox{якщо~} j_1=i_2;\\
      (i_1,j_1-i_2+j_2,-n_0+(F_1\cap (i_2-j_1+F_2))), & \hbox{якщо~} j_1>i_2
    \end{array}
  \right.
\end{align*}
i
\begin{align*}
   \mathfrak{h}(i_1,j_1,F_1)\cdot\mathfrak{h}(i_2,j_2,F_2)&=(i_1,j_1,-n_0+F_1)\cdot(i_2,j_2,-n_0+F_2)=\\
  &=\left\{
    \begin{array}{ll}
       (i_1-j_1+i_2,j_2,(-n_0+j_1-i_2+F_1)\cap(-n_0+F_2)), & \hbox{якщо~} j_1<i_2;\\
       (i_1,j_2,(-n_0+F_1)\cap(-n_0+F_2)),                 & \hbox{якщо~} j_1=i_2;\\
       (i_1,j_1-i_2+j_2,(-n_0+F_1)\cap(-n_0+i_2-j_1+F_2)), & \hbox{якщо~} j_1>i_2
    \end{array}
  \right.= \\
  &=
  \left\{
    \begin{array}{ll}
      (i_1-j_1+i_2,j_2,-n_0+((j_1-i_2+F_1)\cap F_2)), & \hbox{якщо~} j_1<i_2;\\
      (i_1,j_2,-n_0+(F_1\cap F_2)),                   & \hbox{якщо~} j_1=i_2;\\
      (i_1,j_1-i_2+j_2,-n_0+(F_1\cap (i_2-j_1+F_2))), & \hbox{якщо~} j_1>i_2,
    \end{array}
  \right.
\end{align*}
отримуємо, що так означене відображення $\mathfrak{h}\colon (\boldsymbol{B}_{\omega}\times\mathscr{F},\cdot)\to(\boldsymbol{B}_{\omega}\times\mathscr{F}_0,\cdot)$ є гомоморфізмом, а отже, воно є ізоморфізмом. Звідки випливає, що  напівгрупи $\boldsymbol{B}_{\omega}^{\mathscr{F}}$ і $\boldsymbol{B}_{\omega}^{\mathscr{F}_0}$ ізоморфні.
\end{proof}

Врахувавши зауваження \ref{remark-2.2}\eqref{remark-2.2(2)} і \ref{remark-2.2}\eqref{remark-2.2(3)}, надалі для кожної непорожньої множини $F\in\mathscr{F}$ приймемо $n_F=\min F$.

\begin{lemma}\label{lemma-2.3}
Нехай $\mathscr{F}$ --- ${\omega}$-замкнена сім'я індуктивних підмножин у $\omega$ та $F_1$ i $F_2$~--- такі елементи сім'ї $\mathscr{F}$, що $n_{F_1}<n_{F_2}$. Тоді для кожного натурального числа $k\in\{n_{F_1}+1,\ldots,n_{F_2}-1\}$ існує множина $F\in\mathscr{F}$ така, що $F=[k)$.
\end{lemma}

\begin{proof}
З припущення леми випливає, що $F_1\supsetneqq F_2$. Тоді для довільного натурального числа $k\in\{n_{F_1}+1,\ldots,n_{F_2}-1\}$ ціле число $i=n_{F_1}+n_{F_2}-k$ задовольняє умову $n_{F_1}+1\leqslant i\leqslant n_{F_2}-1$, а отже
\begin{equation*}
  (i,i,F_1)\cdot(n_{F_1},n_{F_1},F_2)=(i,i,F_1\cap(n_{F_1}-i+F_2)).
\end{equation*}
Позаяк $F_2\subsetneqq F_1$ і $n_{F_1}+1\leqslant i\leqslant n_{F_2}-1$, то
\begin{equation*}
  \min\{F_1\cap(n_{F_1}-i+F_2)\}=\min\{n_{F_1}-i+F_2\}=n_{F_1}-i+n_{F_2}=k.
\end{equation*}
Отож, виконується рівність
\begin{equation*}
  (i,i,F_1)\cdot(n_{F_1},n_{F_1},F_2)=(i,i,[k)),
\end{equation*}
з якої випливає твердження леми.
\end{proof}

\begin{remark}\label{remark-2.4}
З твердження \ref{proposition-2.1} і леми \ref{lemma-2.3} випливає, якщо $\mathscr{F}$ --- ${\omega}$-замкнена сім'я індуктивних підмножин у $\omega$, то надалі для спрощення викладу будемо вважати, що:
\begin{enumerate}
  \item\label{remark-2.4(1)} $\mathscr{F}=\left\{[k)\colon k\in\omega\right\}$, у випадку нескінченної сім'ї $\mathscr{F}$;
  \item\label{remark-2.4(2)} $\mathscr{F}=\left\{[k)\colon k=0,1,\ldots,n\right\}$ для деякого $n\in\omega$, у випадку скінченної сім'ї $\mathscr{F}$.
\end{enumerate}
\end{remark}

З леми 5 \cite{Gutik-Mykhalenych-2020} врахувавши зауваження \ref{remark-2.4}\eqref{remark-2.4(1)}, отримуємо

\begin{proposition}\label{proposition-2.5}
Якщо $\mathscr{F}$ --- нескінченна ${\omega}$-замкнена сім'я індуктивних підмножин у $\omega$, то діаграма
\begin{equation*}
\tiny{
\xymatrix{
(0,0,\omega)\ar[rd]\ar[dd]&&&&&&&&\\
&(0,0,[1))\ar[ld]\ar[dd]\ar[rd]&&&&&&&\\
(1,1,\omega)\ar[rd]\ar[dd]&&(0,0,[2))\ar[rd]\ar[ld]\ar[dd]&&&&&&\\
&(1,1,[1))\ar[ld]\ar[dd]\ar[rd]&&(0,0,[3))\ar[ld]\ar[dd]\ar[rd]&&&&&\\
(2,2,\omega)\ar[rd]\ar[dd]&&(1,1,[2))\ar[ld]\ar[dd]\ar[rd]&&(0,0,[4))\ar[ld]\ar[dd]\ar[rd]&&&&\\
&(2,2,[1))\ar[ld]\ar[dd]\ar[rd]&&(1,1,[3))\ar[ld]\ar[dd]\ar[rd]&&(0,0,[5))\ar[ld]\ar[dd]\ar[rd]&&&\\
(3,3,\omega)\ar[rd]\ar[dd]&&(2,2,[2))\ar[ld]\ar[dd]\ar[rd]&&(1,1,[4))\ar[ld]\ar[dd]\ar[rd]&&(0,0,[6))\ar[ld]\ar[dd]\ar[rd]&&\\
&(3,3,[1))\ar[ld]\ar[d]\ar[rd]&&(2,2,[3))\ar[ld]\ar[d]\ar[rd]&&(1,1,[5))\ar[ld]\ar[d]\ar[rd]&&(0,0,[7))\ar[ld]\ar[d]\ar[rd]&\\
\cdots&&\cdots&&\cdots&&\cdots&&\cdots
}
}
\end{equation*}
описує природний частковий порядок на в'язці напівгрупи $\boldsymbol{B}_{\omega}^{\mathscr{F}}$.
\end{proposition}

З леми 5 \cite{Gutik-Mykhalenych-2020}, врахувавши зауваження \ref{remark-2.4}\eqref{remark-2.4(2)}, отримуємо

\begin{proposition}\label{proposition-2.9}
Якщо $\mathscr{F}=\left\{[0),\ldots[k)\right\}$ --- скінченна ${\omega}$-замкнена сім'я індуктивних підмножин у $\omega$, то частина діаграми з твердження \ref{proposition-2.5}, що складається з ідемпотентів
\begin{equation*}
\left\{(i,i,[n))\colon i,j\in\omega, \; n=0,1,\ldots, k\right\}
\end{equation*}
напівгрупи $\boldsymbol{B}_{\omega}^{\mathscr{F}}$ та відповідних стрілок, що їх з'єднують,
описує природний частковий порядок на в'язці напівгрупи $\boldsymbol{B}_{\omega}^{\mathscr{F}}$.
\end{proposition}

З означення напівгрупової операції на $\boldsymbol{B}_{\omega}^{\mathscr{F}}$ випливає, якщо $\mathscr{F}$ ---  ${\omega}$-замкнена сім'я підмножин у $\omega$ та $F\in\mathscr{F}$~--- непорожня індуктивна підмножина в $\omega$, то множина
\begin{equation*}
  \boldsymbol{B}_{\omega}^{\{F\}}=\left\{(i,j,F)\colon i,j\in\omega\right\}
\end{equation*}
з індукованою напівгруповою операцією з $\boldsymbol{B}_{\omega}^{\mathscr{F}}$ є піднапівгрупою в $\boldsymbol{B}_{\omega}^{\mathscr{F}}$, яка за твердженням 3 з~\cite{Gutik-Mykhalenych-2020} ізоморфна біциклічній напівгрупі.

\begin{proposition}\label{proposition-2.6}
Нехай $\mathscr{F}$ --- довільна ${\omega}$-замкнена сім'я індуктивних підмножин у $\omega$ та $S$~--- піднапівгрупа в $\boldsymbol{B}_{\omega}^{\mathscr{F}}$, яка ізоморфна біциклічній напівгрупі. Тоді існує така підмножина $F\in\mathscr{F}$, що $S$ --- піднапівгрупа в $\boldsymbol{B}_{\omega}^{\{F\}}$.
\end{proposition}

\begin{proof}
Нехай $\mathfrak{h}\colon\boldsymbol{B}_\omega\to\boldsymbol{B}_{\omega}^{\mathscr{F}}$~--- ізоморфізм вкладення. Тоді за твердженням~1.4.21(2)~\cite{Lawson-1998}, $\mathfrak{h}(0,0)$ i $\mathfrak{h}(1,1)$~--- ідемпотенти напівгрупи $\boldsymbol{B}_{\omega}^{\mathscr{F}}$, а отже за лемою 2 \cite{Gutik-Mykhalenych-2020}, $\mathfrak{h}(0,0)=(i_1,i_1,F_1)$ i $\mathfrak{h}(1,1)=(i_2,i_2,F_2)$ для деяких $i_1,i_2\in\omega$ i $F_1,F_2\in \mathscr{F}$. Також, з нерівності $(1,1)\preccurlyeq(0,0)$ у біциклічній напівгрупі $\boldsymbol{B}_\omega$ випливає, що $(i_2,i_2,F_2)\preccurlyeq(i_1,i_1,F_1)$ у напівгрупі $\boldsymbol{B}_{\omega}^{\mathscr{F}}$. Позаяк
\begin{equation*}
(1,1)\in (0,0)\boldsymbol{B}_{\omega}(0,0)=(0,0)\cdot\boldsymbol{B}_{\omega}\cap \boldsymbol{B}_{\omega}\cdot(0,0),
\end{equation*}
то
\begin{equation*}
(i_2,i_2,F_2)\in (i_1,i_1,F_1)\cdot \boldsymbol{B}_{\omega}^{\mathscr{F}}\cdot(i_1,i_1,F_1)= (i_1,i_1,F_1)\cdot \boldsymbol{B}_{\omega}^{\mathscr{F}}\cap \boldsymbol{B}_{\omega}^{\mathscr{F}}\cdot(i_1,i_1,F_1),
\end{equation*}
а отже з означення напівгрупової операції в $\boldsymbol{B}_{\omega}^{\mathscr{F}}$ випливає, що $i_1\leqslant i_2$. Позаяк $(i_2,i_2,F_2)\preccurlyeq(i_1,i_1,F_1)$ у $\boldsymbol{B}_{\omega}^{\mathscr{F}}$, то з леми 5 \cite{Gutik-Mykhalenych-2020} випливає, що $F_2\subseteq i_1-i_2+F_1$. Також, оскільки $\mathscr{F}$ --- сім'я індуктивних підмножин у $\omega$, то існують такі числа $k_1,k_2\in\omega$, що $F_1=[k_1)$ i $F_2=[k_2)$.

Аналогічно отримуємо, що $\mathfrak{h}(0,1)=(i,j,[a))$  для деяких $i,j,a\in\omega$. Позаяк $(1,0)$~--- інверсний елемент до $(0,1)$ в напівгрупі $\boldsymbol{B}_\omega$, то за твердженням 1.4.21 \cite{Lawson-1998} і лемою 4 \cite{Gutik-Mykhalenych-2020} маємо, що
\begin{equation*}
  \mathfrak{h}(1,0)=\mathfrak{h}\left((1,0)^{-1}\right)=\left(\mathfrak{h}(1,0)\right)^{-1}=\left((i,j,[a))\right)^{-1}=(j,i,[a)),
\end{equation*}
Тоді
\begin{equation*}
  (i_1,i_1,[k_1))=\mathfrak{h}(0,0)=\mathfrak{h}((0,1)\cdot(1,0))=\mathfrak{h}(0,1)\cdot\mathfrak{h}(1,0)=(i,j,[a))\cdot(j,i,[a))=(i,i,[a))
\end{equation*}
і, аналогічно,
\begin{equation*}
  (i_2,i_2,[k_2))=\mathfrak{h}(1,1)=\mathfrak{h}((1,0)\cdot(0,1))=\mathfrak{h}(1,0)\cdot\mathfrak{h}(0,1)=(j,i,[a))\cdot(i,j,[a))=(j,j,[a)).
\end{equation*}
З однозначності зображення елементів напівгрупи $\boldsymbol{B}_{\omega}^{\mathscr{F}}$ випливає, що $i=i_1$, $j=i_2$ й ${a=k_1=k_2}$, а отже  $\mathfrak{h}(0,1)=(i_1,i_2,[k_1))$ i $\mathfrak{h}(1,0)=(i_2,i_1,[k_1))$. З означення напівгрупової операції в $\boldsymbol{B}_{\omega}^{\mathscr{F}}$ випливає, що піднапівгрупа $S=\left\langle(i_1,i_2,[k_1)),(i_2,i_1,[k_1))\right\rangle$ в $\boldsymbol{B}_{\omega}^{\mathscr{F}}$, породжена елементами $(i_1,i_2,[k_1))$ i $(i_2,i_1,[k_1))$, є інверсною піднапівгрупою в $\boldsymbol{B}_{\omega}^{\{[k_1)\}}$. Напівгрупа $\boldsymbol{B}_{\omega}^{\{[k_1)\}}$ шукана.
\end{proof}

Теорема \ref{theorem-2.7} описує групові конгруенції на напівгрупі $\boldsymbol{B}_{\omega}^{\mathscr{F}}$.

\begin{theorem}\label{theorem-2.7}
Нехай $\mathscr{F}$ --- довільна ${\omega}$-замкнена сім'я індуктивних непорожніх підмножин у $\omega$ та $\mathfrak{C}$~--- конгруенція на $\boldsymbol{B}_{\omega}^{\mathscr{F}}$. Тоді такі умови еквівалентні:
\begin{enumerate}
  \item\label{theorem-2.7(1)} $\mathfrak{C}$~--- групова конгруенція на $\boldsymbol{B}_{\omega}^{\mathscr{F}}$;
  \item\label{theorem-2.7(2)} існує піднапівгрупа $S$ в $\boldsymbol{B}_{\omega}^{\mathscr{F}}$, яка ізоморфна біциклічній напівгрупі та два різні елементи напівгрупи $S$ є $\mathfrak{C}$-ек\-вівалентними;
  \item\label{theorem-2.7(3)} для довільної піднапівгрупи $T$ в $\boldsymbol{B}_{\omega}^{\mathscr{F}}$, яка ізоморфна біциклічній напівгрупі, два різні елементи напівгрупи $T$ є $\mathfrak{C}$-ек\-вівалентними.
\end{enumerate}
\end{theorem}

\begin{proof}
Імплікації \eqref{theorem-2.7(1)}$\Rightarrow$\eqref{theorem-2.7(3)} i \eqref{theorem-2.7(3)}$\Rightarrow$\eqref{theorem-2.7(2)} очевидні.

Доведемо імплікацію \eqref{theorem-2.7(2)}$\Rightarrow$\eqref{theorem-2.7(1)}. Нехай $S$ --- піднапівгрупа в $\boldsymbol{B}_{\omega}^{\mathscr{F}}$, яка ізоморфна біциклічній напівгрупі та два різні елементи в $S$ є $\mathfrak{C}$-ек\-вівалентними. За твердженням \ref{proposition-2.6} існує множина $F\in\mathscr{F}$ така, що $S$ --- піднапівгрупа в $\boldsymbol{B}_{\omega}^{\{F\}}$ і за твердженням 3 \cite{Gutik-Mykhalenych-2020} напівгрупа $\boldsymbol{B}_{\omega}^{\{F\}}$ ізоморфна біциклічній напівгрупі. Тоді два різні елементи в $\boldsymbol{B}_{\omega}^{\{F\}}$ є $\mathfrak{C}$-ек\-вівалентними, а за наслідком 1.32 \cite{Clifford-Preston-1961} усі ідемпотенти напівгрупа $\boldsymbol{B}_{\omega}^{\{F\}}$ є $\mathfrak{C}$-ек\-вівалентними.

З твердження~\ref{proposition-2.1} випливає, що не змешуючи загальності, можемо вважати, що $[0)\in\mathscr{F}$. 
За тверд\-жен\-ням \ref{proposition-2.5} існують два різні ідемпотенти $(i_1,i_1,[0)),(i_2,i_2,[0))$ у піднапівгрупі $\boldsymbol{B}_{\omega}^{\{[0)\}}$, які є $\mathfrak{C}$-ек\-вівалентними. Далі, знову використавши наслідок 1.32 \cite{Clifford-Preston-1961}, отримуємо, що всі ідемпотенти піднапівгрупи $\boldsymbol{B}_{\omega}^{\{[0)\}}$ є $\mathfrak{C}$-ек\-вівалентними. З тверджень \ref{proposition-2.5} і \ref{proposition-2.9} випливає, що для довільних ідемпотентів $(i,i,[a)),(j,j,[b))\in\boldsymbol{B}_{\omega}^{\mathscr{F}}$ існують ідемпотенти $(i_1,i_1,[0)),(i_2,i_2,[0))\in \boldsymbol{B}_{\omega}^{\{[0)\}}$ такі, що $(i_1,i_1,[0))\preccurlyeq(i,i,[a))\preccurlyeq(i_2,i_2,[0))$ i $(i_1,i_1,[0))\preccurlyeq(j,j,[b))\preccurlyeq(i_2,i_2,[0))$. Звідки випливає, що ідемпотенти $(i,i,[a))$ та $(j,j,[b))$ також є $\mathfrak{C}$-ек\-вівалентними, а отже виконується імплікація \eqref{theorem-2.7(2)}$\Rightarrow$\eqref{theorem-2.7(1)}.
\end{proof}

Зауважимо, що аналогічне твердження до теореми \ref{theorem-2.7} справджується і для напівгрупи часткових коскiнченних iзометрiй натуральних чисел $\mathbf{I}\mathbb{N}_{\infty}$ (див. \cite[теорема 9]{Gutik-Savchuk-2018}).

З прикладу \ref{example-2.8} випливає, що на напівгрупі $\boldsymbol{B}_{\omega}^{\mathscr{F}}$ існують негрупові конгруенції.

\begin{example}\label{example-2.8}
Нехай $\mathscr{F}$ --- довільна ${\omega}$-замкнена сім'я індуктивних непорожніх підмножин у $\omega$. Означимо відображення $\mathfrak{h}\colon \boldsymbol{B}_{\omega}^{\mathscr{F}}\to \boldsymbol{B}_{\omega}$ за формулою
\begin{equation*}
\mathfrak{h}(i,j,F)=(i,j),\qquad i,j\in\omega,\quad F\in\mathscr{F}.
\end{equation*}
З оз\-на\-чен\-ня напівгрупових операцій на напівгрупах $\boldsymbol{B}_{\omega}^{\mathscr{F}}$ i $\boldsymbol{B}_{\omega}$ випливає, що так означене відоб\-ра\-ження $\mathfrak{h}$ є сюр'єктивним гомоморфізмом. Отож, конгруенція $\mathfrak{h}^\sharp$ на напівгрупі $\boldsymbol{B}_{\omega}^{\mathscr{F}}$, породжена гомомор\-фіз\-мом $\mathfrak{h}$, не є груповою.
\end{example}


\section{Гомоморфні ретракти та ізоморфізми напівгрупи $\boldsymbol{B}_{\omega}^{\mathscr{F}}$}\label{section-3}

\begin{example}\label{example-3.1}
Нехай $\mathscr{F}$ --- довільна ${\omega}$-замкнена сім'я індуктивних непорожніх підмножин у $\omega$. Зафіксуємо довільну множину $F\in\mathscr{F}$. Означимо відображення $\mathfrak{h}_F\colon \boldsymbol{B}_{\omega}\to \boldsymbol{B}_{\omega}^{\mathscr{F}}$ за формулою
\begin{equation*}
\mathfrak{h}_F(i,j)=(i,j,F),\qquad i,j\in\omega.
\end{equation*}
З означення напівгрупових операцій на напівгрупах $\boldsymbol{B}_{\omega}^{\mathscr{F}}$ i $\boldsymbol{B}_{\omega}$ випливає, що так означене відображення $\mathfrak{h}_F$ є гомоморфізмом.
\end{example}

Оскільки композиція гомоморфізмів напівгруп є гомоморфізмом, то з визначення гомо\-мор\-фіз\-мів $\mathfrak{h}\colon \boldsymbol{B}_{\omega}^{\mathscr{F}}\to \boldsymbol{B}_{\omega}$ і  $\mathfrak{h}_F\colon \boldsymbol{B}_{\omega}\to \boldsymbol{B}_{\omega}^{\mathscr{F}}$ (див. приклади \ref{example-2.8} і \ref{example-3.1}, відповідно), випливає, що їх композиція $\mathfrak{h}_F\circ\mathfrak{h}$ є гомоморфною ретракцією напівгрупи $\boldsymbol{B}_{\omega}^{\mathscr{F}}$, тобто виконується таке тверд\-жен\-ня:

\begin{proposition}\label{proposition-3.2}
Нехай $\mathscr{F}$ --- довільна ${\omega}$-замкнена сім'я індуктивних непорожніх підмножин у $\omega$ та $F\in\mathscr{F}$. Тоді піднапівгрупа $\boldsymbol{B}_{\omega}^{\{F\}}$ є гомоморфним ретрактом напівгрупи $\boldsymbol{B}_{\omega}^{\mathscr{F}}$.
\end{proposition}

\begin{example}\label{example-3.3}
Нехай $\mathscr{F}$ --- довільна ${\omega}$-замкнена сім'я індуктивних непорожніх підмножин у $\omega$. Зафіксуємо довільну множину $[k)\in\mathscr{F}$. Означимо відображення $\mathfrak{h}_k\colon \boldsymbol{B}_{\omega}^{\mathscr{F}}\to \boldsymbol{B}_{\omega}^{\mathscr{F}}$ за формулою
\begin{equation*}
  \mathfrak{h}_k(i,j,[a))=
  \left\{
    \begin{array}{ll}
      (i,j,[k)), & \hbox{якщо~} a\leqslant k;\\
      (i,j,[a)), & \hbox{якщо~} a>k,
    \end{array}
  \right.
\qquad i,j\in\omega, \quad [a)\in \mathscr{F}.
\end{equation*}
\end{example}

\begin{lemma}\label{lemma-3.4}
Якщо $\mathscr{F}$ --- довільна ${\omega}$-замкнена сім'я індуктивних непорожніх підмножин у $\omega$ та $[k)\in\mathscr{F}$, то $\mathfrak{h}_k\colon \boldsymbol{B}_{\omega}^{\mathscr{F}}\to \boldsymbol{B}_{\omega}^{\mathscr{F}}$ --- гомоморфізм.
\end{lemma}

\begin{proof}
Нехай $(i_1,j_1,[a_1))$ i $(i_2,j_2,[a_2))$~--- довільні елементи напівгрупи $\boldsymbol{B}_{\omega}^{\mathscr{F}}$. Розглянемо мож\-ливі випадки.

\smallskip

(1) Нехай $a_1,a_2\leqslant k$. Тоді
\begin{align*}
  \mathfrak{h}_k((i_1,j_1,[a_1))\cdot(i_2,j_2,[a_2)))&=
 \left\{
   \begin{array}{ll}
     \mathfrak{h}_k(i_1-j_1+i_2,j_2,(j_1-i_2+[a_1))\cap[a_2)), & \hbox{якщо~} j_1<i_2;\\
     \mathfrak{h}_k(i_1,j_2,[a_1)\cap[a_2)),                   & \hbox{якщо~} j_1=i_2;\\
     \mathfrak{h}_k(i_1,j_1-i_2+j_2,[a_1)\cap(i_2-j_1+[a_2))), & \hbox{якщо~} j_1>i_2
   \end{array}
 \right.
\\
  &=
\left\{
   \begin{array}{ll}
     (i_1-j_1+i_2,j_2,[k)), & \hbox{якщо~} j_1<i_2;\\
     (i_1,j_2,[k)),         & \hbox{якщо~} j_1=i_2;\\
     (i_1,j_1-i_2+j_2,[k)), & \hbox{якщо~} j_1>i_2
   \end{array}
 \right.
\end{align*}
i
\begin{align*}
  \mathfrak{h}_k(i_1,j_1,[a_1))\cdot\mathfrak{h}_k(i_2,j_2,[a_2))&=(i_1,j_1,[k))\cdot(i_2,j_2,[k))=\\
  &=
\left\{
   \begin{array}{ll}
     (i_1-j_1+i_2,j_2,(j_1-i_2+[k))\cap[k)), & \hbox{якщо~} j_1<i_2;\\
     (i_1,j_2,[k)\cap[k)),                   & \hbox{якщо~} j_1=i_2;\\
     (i_1,j_1-i_2+j_2,[k)\cap(i_2-j_1+[k))), & \hbox{якщо~} j_1>i_2
   \end{array}
 \right.
=\\
  &=
\left\{
   \begin{array}{ll}
     (i_1-j_1+i_2,j_2,[k)), & \hbox{якщо~} j_1<i_2;\\
     (i_1,j_2,[k)),         & \hbox{якщо~} j_1=i_2;\\
     (i_1,j_1-i_2+j_2,[k))), & \hbox{якщо~} j_1>i_2
   \end{array}
 \right.
\end{align*}

\smallskip

(2) Припустимо, що $a_1>k$ i $a_2\leqslant k$. Тоді маємо, що
\begin{align*}
  \mathfrak{h}_k((i_1,j_1,[a_1))\cdot(i_2,j_2,[a_2)))&=
 \left\{
   \begin{array}{ll}
     \mathfrak{h}_k(i_1-j_1+i_2,j_2,(j_1-i_2+[a_1))\cap[a_2)), & \hbox{якщо~} j_1<i_2;\\
     \mathfrak{h}_k(i_1,j_2,[a_1)\cap[a_2)),                   & \hbox{якщо~} j_1=i_2;\\
     \mathfrak{h}_k(i_1,j_1-i_2+j_2,[a_1)\cap(i_2-j_1+[a_2))), & \hbox{якщо~} j_1>i_2
   \end{array}
 \right.
\\
  &=
\left\{
   \begin{array}{ll}
     (i_1-j_1+i_2,j_2,[k)),           & \hbox{якщо~} j_1<i_2 \hbox{~i~} j_1-i_2+a_1\leqslant k;\\
     (i_1-j_1+i_2,j_2,j_1-i_2+[a_1)), & \hbox{якщо~} j_1<i_2 \hbox{~i~} j_1-i_2+a_1>k;\\
     (i_1,j_2,[a_1)),                 & \hbox{якщо~} j_1=i_2;\\
     (i_1,j_1-i_2+j_2,[a_1)),         & \hbox{якщо~} j_1>i_2
   \end{array}
 \right.
\end{align*}
i
\begin{align*}
  \mathfrak{h}_k(i_1,j_1,[a_1))\cdot\mathfrak{h}_k(i_2,j_2,[a_2))&=(i_1,j_1,[a_1))\cdot(i_2,j_2,[k))=\\
  &=
\left\{
   \begin{array}{ll}
     (i_1-j_1+i_2,j_2,(j_1-i_2+[a_1))\cap[k)), & \hbox{якщо~} j_1<i_2;\\
     (i_1,j_2,[a_1)\cap[k)),                   & \hbox{якщо~} j_1=i_2;\\
     (i_1,j_1-i_2+j_2,[a_1)\cap(i_2-j_1+[k))), & \hbox{якщо~} j_1>i_2
   \end{array}
 \right.
=\\
  &=
\left\{
   \begin{array}{ll}
     (i_1-j_1+i_2,j_2,[k)),           & \hbox{якщо~} j_1<i_2 \hbox{~i~} j_1-i_2+a_1\leqslant k;\\
     (i_1-j_1+i_2,j_2,j_1-i_2+[a_1)), & \hbox{якщо~} j_1<i_2 \hbox{~i~} j_1-i_2+a_1>k;\\
     (i_1,j_2,[a_1)),                 & \hbox{якщо~} j_1=i_2;\\
     (i_1,j_1-i_2+j_2,[a_1)),         & \hbox{якщо~} j_1>i_2.
   \end{array}
 \right.
\end{align*}

\smallskip

(3) Якщо $a_1\leqslant k$ і $a_2>k$, то
\begin{align*}
  \mathfrak{h}_k((i_1,j_1,[a_1))\cdot(i_2,j_2,[a_2)))&=
 \left\{
   \begin{array}{ll}
     \mathfrak{h}_k(i_1-j_1+i_2,j_2,(j_1-i_2+[a_1))\cap[a_2)), & \hbox{якщо~} j_1<i_2;\\
     \mathfrak{h}_k(i_1,j_2,[a_1)\cap[a_2)),                   & \hbox{якщо~} j_1=i_2;\\
     \mathfrak{h}_k(i_1,j_1-i_2+j_2,[a_1)\cap(i_2-j_1+[a_2))), & \hbox{якщо~} j_1>i_2
   \end{array}
 \right.
\\
  &=
\left\{
   \begin{array}{ll}
     (i_1-j_1+i_2,j_2,[a_2)),         & \hbox{якщо~} j_1<i_2;\\
     (i_1,j_2,[a_2)),                 & \hbox{якщо~} j_1=i_2;\\
     (i_1,j_1-i_2+j_2,[k)),           & \hbox{якщо~} j_1>i_2  \hbox{~i~} j_1-i_2+a_1\leqslant k;\\
     (i_1,j_1-i_2+j_2,i_2-j_1+[a_2)), & \hbox{якщо~} j_1>i_2  \hbox{~i~} j_1-i_2+a_1>k
   \end{array}
 \right.
\end{align*}
i
\begin{align*}
  \mathfrak{h}_k(i_1,j_1,[a_1))\cdot\mathfrak{h}_k(i_2,j_2,[a_2))&=(i_1,j_1,[k))\cdot(i_2,j_2,[a_2))=\\
  &=
\left\{
   \begin{array}{ll}
     (i_1-j_1+i_2,j_2,(j_1-i_2+[k))\cap[a_2)), & \hbox{якщо~} j_1<i_2;\\
     (i_1,j_2,[k)\cap[a_2)),                   & \hbox{якщо~} j_1=i_2;\\
     (i_1,j_1-i_2+j_2,[k)\cap(i_2-j_1+[a_2))), & \hbox{якщо~} j_1>i_2
   \end{array}
 \right.
=\\
  &=
\left\{
   \begin{array}{ll}
     (i_1-j_1+i_2,j_2,[a_2)),         & \hbox{якщо~} j_1<i_2;\\
     (i_1,j_2,[a_2)),                 & \hbox{якщо~} j_1=i_2;\\
     (i_1,j_1-i_2+j_2,[k)),           & \hbox{якщо~} j_1>i_2  \hbox{~i~} j_1-i_2+a_1\leqslant k;\\
     (i_1,j_1-i_2+j_2,i_2-j_1+[a_2)), & \hbox{якщо~} j_1>i_2  \hbox{~i~} j_1-i_2+a_1>k.
   \end{array}
 \right.
\end{align*}

\smallskip

(4) Припустимо, що $a_1,a_2>k$. Тоді маємо, що
\begin{align*}
  \mathfrak{h}_k((i_1,j_1,[a_1))\cdot(i_2,j_2,[a_2)))&=
 \left\{
   \begin{array}{ll}
     \mathfrak{h}_k(i_1-j_1+i_2,j_2,(j_1-i_2+[a_1))\cap[a_2)), & \hbox{якщо~} j_1<i_2;\\
     \mathfrak{h}_k(i_1,j_2,[a_1)\cap[a_2)),                   & \hbox{якщо~} j_1=i_2;\\
     \mathfrak{h}_k(i_1,j_1-i_2+j_2,[a_1)\cap(i_2-j_1+[a_2))), & \hbox{якщо~} j_1>i_2
   \end{array}
 \right.
\\
  &=
\left\{
   \begin{array}{ll}
     (i_1-j_1+i_2,j_2,[a_2)),         & \hbox{якщо~} j_1<i_2 \hbox{~i~} j_1-i_2+a_1\leqslant a_2;\\
     (i_1-j_1+i_2,j_2,j_1-i_2+[a_1)), & \hbox{якщо~} j_1<i_2 \hbox{~i~} j_1-i_2+a_1>a_2;\\
     (i_1,j_2,[a_2)),                 & \hbox{якщо~} j_1=i_2;\\
     (i_1,j_1-i_2+j_2,i_2-j_1+[a_2)), & \hbox{якщо~} j_1>i_2 \hbox{~i~} i_2-j_1+a_2\geqslant a_1;\\
     (i_1,j_1-i_2+j_2,[a_1)),         & \hbox{якщо~} j_1>i_2 \hbox{~i~} i_2-j_1+a_2<a_1
   \end{array}
 \right.
\end{align*}
i
\begin{align*}
  \mathfrak{h}_k(i_1,j_1,[a_1))\cdot\mathfrak{h}_k(i_2,j_2,[a_2))&=(i_1,j_1,[a_1))\cdot(i_2,j_2,[a_2))=\\
  &=
\left\{
   \begin{array}{ll}
     (i_1-j_1+i_2,j_2,(j_1-i_2+[a_1))\cap[a_2)), & \hbox{якщо~} j_1<i_2;\\
     (i_1,j_2,[a_1)\cap[a_2)),                   & \hbox{якщо~} j_1=i_2;\\
     (i_1,j_1-i_2+j_2,[a_1)\cap(i_2-j_1+[a_2))), & \hbox{якщо~} j_1>i_2
   \end{array}
 \right.
=\\
  &=
\left\{
   \begin{array}{ll}
     (i_1-j_1+i_2,j_2,[a_2)),         & \hbox{якщо~} j_1<i_2 \hbox{~i~} j_1-i_2+a_1\leqslant a_2;\\
     (i_1-j_1+i_2,j_2,j_1-i_2+[a_1)), & \hbox{якщо~} j_1<i_2 \hbox{~i~} j_1-i_2+a_1>a_2;\\
     (i_1,j_2,[a_2)),                 & \hbox{якщо~} j_1=i_2;\\
     (i_1,j_1-i_2+j_2,i_2-j_1+[a_2)), & \hbox{якщо~} j_1>i_2 \hbox{~i~} i_2-j_1+a_2\geqslant a_1;\\
     (i_1,j_1-i_2+j_2,[a_1)),         & \hbox{якщо~} j_1>i_2 \hbox{~i~} i_2-j_1+a_2<a_1.
   \end{array}
 \right.
\end{align*}

Отож,  відображення $\mathfrak{h}_k\colon \boldsymbol{B}_{\omega}^{\mathscr{F}}\to \boldsymbol{B}_{\omega}^{\mathscr{F}}$ є гомоморфізмом.
\end{proof}

Позаяк
\begin{equation*}
  \mathfrak{h}_k(\boldsymbol{B}_{\omega}^{\mathscr{F}})=\boldsymbol{B}_{\omega}^{\mathscr{F}_{\geqslant k}}=\left\{(i,j,[a))\in\boldsymbol{B}_{\omega}^{\mathscr{F}}\colon [a)\in \mathscr{F} \hbox{~~i~~} a\geqslant k\right\}
\end{equation*}
--- інверсна піднапівгрупа в $\boldsymbol{B}_{\omega}^{\mathscr{F}}$, яка є множиною нерухомих точок гомоморфізму $\mathfrak{h}_k\colon \boldsymbol{B}_{\omega}^{\mathscr{F}}\to \boldsymbol{B}_{\omega}^{\mathscr{F}}$, то справджується

\begin{proposition}\label{proposition-3.5}
Нехай $\mathscr{F}$ --- довільна ${\omega}$-замкнена сім'я індуктивних непорожніх підмножин у $\omega$ та $[k)\in\mathscr{F}$ для деякого числа $k\in\omega$. Тоді піднапівгрупа $\boldsymbol{B}_{\omega}^{\mathscr{F}_{\geqslant k}}$ є гомоморфним ретрактом напівгрупи $\boldsymbol{B}_{\omega}^{\mathscr{F}}$.
\end{proposition}

Нехай $\mathscr{F}=\left\{[0), [1),\ldots,[k)\right\}$ ---  сім'я  підмножин у $\omega$, де $k\geqslant 1$. Для довільного натурального числа $n<k$ означимо
\begin{equation*}
  \boldsymbol{B}_{\omega}^{\mathscr{F}_{\leqslant n}}= \left\{(i,j,[a))\in\boldsymbol{B}_{\omega}^{\mathscr{F}}\colon [a)\in \mathscr{F} \hbox{~~i~~} a\leqslant n\right\}.
\end{equation*}
Очевидно, що $\boldsymbol{B}_{\omega}^{\mathscr{F}_{\leqslant n}}$ --- інверсна піднапівгрупа в $\boldsymbol{B}_{\omega}^{\mathscr{F}}$.

З твердження \ref{proposition-3.2} випливає, що для сім'ї $\mathscr{F}=\left\{[0), [1)\right\}$ піднапівгрупа $\boldsymbol{B}_{\omega}^{\{[0)\}}$ є гомоморфним рет\-рак\-том напівгрупи $\boldsymbol{B}_{\omega}^{\mathscr{F}}$. Однак, з подальших тверджень випливає, що для довільного натурального числа $k\geqslant 2$ піднапівгрупа $\boldsymbol{B}_{\omega}^{\mathscr{F}_{\leqslant n}}$ не є гомоморфним ретрактом напівгрупи $\boldsymbol{B}_{\omega}^{\mathscr{F}}$ для довільного натурального числа $n<k$.

\begin{theorem}\label{theorem-3.6}
Нехай $\mathscr{F}$ --- довільна ${\omega}$-замкнена сім'я індуктивних непорожніх підмножин у $\omega$ та $\omega\in \mathscr{F}$. Тоді:
\begin{enumerate}
  \item\label{theorem-3.6(1)} якщо сім'я $\mathscr{F}$ --- нескінченна, то для довільного натурального числа $k$ напівгрупа $\boldsymbol{B}_{\omega}^{\mathscr{F}_{\leqslant k}}$ не є гомоморфним ретрактом напівгрупи $\boldsymbol{B}_{\omega}^{\mathscr{F}}$;
  \item\label{theorem-3.6(2)} якщо сім'я $\mathscr{F}$ --- скінченна та $\displaystyle\bigcap\mathscr{F}=[t)$ для деякого натурального числа $t\geqslant 2$, то для довільного натурального числа $k<t$ напівгрупа $\boldsymbol{B}_{\omega}^{\mathscr{F}_{\leqslant k}}$ не є гомоморфним ретрактом на\-пів\-гру\-пи $\boldsymbol{B}_{\omega}^{\mathscr{F}}$.
\end{enumerate}
\end{theorem}

\begin{proof}
Припустимо протилежне: хоча б в одному з випадків \eqref{theorem-3.6(1)}, чи \eqref{theorem-3.6(2)} напівгрупа $\boldsymbol{B}_{\omega}^{\mathscr{F}_{\leqslant n}}$ є гомоморфним ретрактом напівгрупи $\boldsymbol{B}_{\omega}^{\mathscr{F}}$.

Позаяк звуження напівгрупового гомоморфізму $h\colon S\to S$ на піднапівгрупу $T\subseteq S$ є знову гомоморфізмом, то, припустивши, що $\boldsymbol{B}_{\omega}^{\mathscr{F}_{\leqslant k}}$ є гомоморфним ретрактом напівгрупи $\boldsymbol{B}_{\omega}^{\mathscr{F}}$, який породжується деяким гомоморфізмом $\mathfrak{h}\colon \boldsymbol{B}_{\omega}^{\mathscr{F}}\to \boldsymbol{B}_{\omega}^{\mathscr{F}}$, отримуємо, що піднапівгрупа $\boldsymbol{B}_{\omega}^{\mathscr{F}_{\leqslant k}}$ є гомоморфним ретрактом напівгрупи $\boldsymbol{B}_{\omega}^{\mathscr{F}_{\leqslant k+1}}$, який породжується деякою гомоморфною ретракцією $\mathfrak{h}^k\colon \boldsymbol{B}_{\omega}^{\mathscr{F}_{\leqslant k+1}}\to \boldsymbol{B}_{\omega}^{\mathscr{F}_{\leqslant k}}$.

З визначення природного часткового порядку на в'язці $E(\boldsymbol{B}_{\omega}^{\mathscr{F}_{\leqslant k+1}})$ (див. твердження \ref{proposition-2.5} і \ref{proposition-2.9}) випливає, що
\begin{equation*}
  (1,1,[k))\preccurlyeq(0,0,[k+1))\preccurlyeq(0,0,[k)).
\end{equation*}
Позаяк гомоморфізм інверсних напівгруп зберігає природний частковий порядок на їх в'язках у бік образу (див. \cite[твердження~1.4.21(6)]{Lawson-1998}), то
\begin{equation}\label{eq-3.1}
  (1,1,[k))\preccurlyeq \mathfrak{h}^k(0,0,[k+1))\preccurlyeq(0,0,[k)).
\end{equation}
З нерівностей \eqref{eq-3.1} та визначення природного часткового порядку на $E(\boldsymbol{B}_{\omega}^{\mathscr{F}_{\leqslant k+1}})$ (див. твердження \ref{proposition-2.5} і \ref{proposition-2.9}) випливає, що виконується лише один з випадків:
\begin{enumerate}
  \item\label{theorem-3.6-proof(1)} $\mathfrak{h}^k(0,0,[k+1))=(0,0,[k))$;
  \item\label{theorem-3.6-proof(2)} $\mathfrak{h}^k(0,0,[k+1))=(1,1,[k))$;
  \item\label{theorem-3.6-proof(3)} $\mathfrak{h}^k(0,0,[k+1))=(1,1,[k-1))$.
\end{enumerate}

\medskip

Припустимо, що виконується випадок \eqref{theorem-3.6-proof(1)}: $\mathfrak{h}^k(0,0,[k+1))=(0,0,[k))$. З твердження \ref{proposition-2.6} випливає, що $\mathfrak{h}^k(i,j,[k+1))\in \boldsymbol{B}_{\omega}^{\{[k)\}}$ для всіх $i,j\in\omega$, оскільки за твердженням~3 з \cite{Gutik-Mykhalenych-2020} напівгрупа $\boldsymbol{B}_{\omega}^{\{[k+1)\}}$ ізоморфна біциклічній напівгрупі. Також з тверджень \ref{proposition-2.5} і \ref{proposition-2.9} випливає, що $\mathfrak{h}^k(1,1,[k+1))\preccurlyeq(1,1,[k))$.

За твердженням~1.4.21(2) з~\cite{Lawson-1998}, $\mathfrak{h}^k(1,1,[k+1))$~--- ідемпотент напівгрупи $\boldsymbol{B}_{\omega}^{\mathscr{F}}$, як гомоморфний образ ідемпотента $(1,1,[k+1))$, а отже, $\mathfrak{h}^k(1,1,[k+1))=(i,i,[k))$ для деякого натурального чис\-ла~$i$. Позаяк
\begin{equation*}
  (0,0,[k+1))=(0,1,[k+1))\cdot(1,0,[k+1)),
\end{equation*}
то за твердженням~1.4.21(1) \cite{Lawson-1998} маємо, що
\begin{align*}
  (0,0,[k))&=\mathfrak{h}^k(0,0,[k+1))= \\
  &=\mathfrak{h}^k((0,1,[k+1))\cdot(1,0,[k+1)))=\\
  &=\mathfrak{h}^k(0,1,[k+1))\cdot\mathfrak{h}^k(1,0,[k+1))=\\
  &=\mathfrak{h}^k(0,1,[k+1))\cdot\mathfrak{h}^k((0,1,[k+1))^{-1})=\\
  &=\mathfrak{h}^k(0,1,[k+1))\cdot(\mathfrak{h}^k(0,1,[k+1)))^{-1},
\end{align*}
і, прийнявши $\mathfrak{h}^k(0,1,[k+1))=(m,n,[k))$, за лемою 4 з~\cite{Gutik-Mykhalenych-2020} отримуємо, що
\begin{equation*}
  (m,n,[k))\cdot(m,n,[k))^{-1}=(m,n,[k))\cdot(n,m,[k))=(m,m,[k))=(0,0,[k)),
\end{equation*}
а отже $\mathfrak{h}^k(0,1,[k+1))=(0,n,[k))$, для деякого натурального числа $n$. Аналогічно з рівностей
\begin{align*}
  (i,i,[k))&= \mathfrak{h}^k(1,1,[k+1))=\\
  &=\mathfrak{h}^k((1,0,[k+1))\cdot(0,1,[k+1)))=\\
  &=\mathfrak{h}^k(1,0,[k+1)))\cdot\mathfrak{h}^k(0,1,[k+1))=\\
  &=\mathfrak{h}^k((0,1,[k+1))^{-1})\cdot\mathfrak{h}^k(0,1,[k+1))=\\
  &=(\mathfrak{h}^k(0,1,[k+1)))^{-1}\cdot\mathfrak{h}^k(0,1,[k+1))=\\
  &=(0,n,[k))^{-1}\cdot(0,n,[k))=\\
  &=(n,0,[k))\cdot(0,n,[k))=\\
  &=(n,n,[k))
\end{align*}
випливає, що $n=i$.

Припустимо, що $i\geqslant 2$. Тоді з
\begin{align*}
  (3,3,[k))&\preccurlyeq(2,2,[k+1))=\\
  &=(2,0,[k+1))\cdot(0,2,[k+1))=\\
  &=(1,0,[k+1))^2\cdot(0,1,[k+1))^2
\end{align*}
випливає, що
\begin{align*}
  (3,3,[k))&=\mathfrak{h}^k(3,3,[k))\preccurlyeq\\
  &\preccurlyeq\mathfrak{h}^k(2,2,[k+1))=\\
  &=\mathfrak{h}^k((1,0,[k+1))^2\cdot(0,1,[k+1))^2)=\\
  &=(\mathfrak{h}^k(1,0,[k+1)))^2\cdot(\mathfrak{h}^k(0,1,[k+1)))^2=\\
  &=(i,0,[k))^2\cdot(0,i,[k))^2=\\
  &=(2i,0,[k))\cdot(0,2i,[k))=\\
  &=(2i,2i,[k)),
\end{align*}
а це суперечить визначенню природного часткового порядку на напівґратці $E(\boldsymbol{B}_{\omega}^{\mathscr{F}_{\leqslant k+1}})$ (див. тверд\-жен\-ня \ref{proposition-2.5} і \ref{proposition-2.9}). Отож, отримуємо, що $\mathfrak{h}^k(0,1,[k+1))=(0,1,[k))$ і за твердженням~1.4.21(1) з \cite{Lawson-1998}, $\mathfrak{h}^k(1,0,[k+1))=(1,0,[k))$. Тоді з означення напівгрупової операції на $\boldsymbol{B}_{\omega}^{\mathscr{F}}$ випливає, що 
\begin{equation*}
\mathfrak{h}^k(p,q,[k+1))=(p,q,[k)),
\end{equation*}
для всіх $p,q\in\omega$.

Припустимо, що $i>j$. Тоді для $a=0,1,\ldots,k$ маємо
\begin{align*}
  \mathfrak{h}^k((i,i,[a))\cdot(j,j,[k+1)))&=\mathfrak{h}^k(i,i,[a)\cap(j-i+[k+1)))= \\
  &=
  \left\{
    \begin{array}{ll}
      \mathfrak{h}^k(i,i,j-i+[k+1)), & \hbox{якщо~} 0\leqslant a<j-i+k+1;\\
      \mathfrak{h}^k(i,i,[a)),       & \hbox{якщо~} a\geqslant j-i+k+1
    \end{array}
  \right.
  =\\
  &=
  \left\{
    \begin{array}{ll}
      (i,i,j-i+[k+1)), & \hbox{якщо~} 0\leqslant a<j-i+k+1;\\
      (i,i,[a)),       & \hbox{якщо~} a\geqslant j-i+k+1
    \end{array}
  \right.
\end{align*}
i
\begin{align*}
  \mathfrak{h}^k(i,i,[a))\cdot\mathfrak{h}^k(j,j,[k+1))&=(i,i,[a))\cdot(j,j,[k))= \\
  &=(i,i,[a)\cap(j-i+[k)))=\\
  &=
  \left\{
    \begin{array}{ll}
      (i,i,j-i+[k)), & \hbox{якщо~} 0\leqslant a<j-i+k;\\
      (i,i,[a)),       & \hbox{якщо~} a\geqslant j-i+k.
    \end{array}
  \right.
\end{align*}
Отож, якщо $a=j-i+k$ (а такий випадок завжди можливий зокрема коли $a=-1+k$ тобто $i=j+1$), то отримуємо, що
\begin{equation*}
  \mathfrak{h}^k((i,i,[a))\cdot(j,j,[k+1)))=(i,i,j-i+[k+1))=(i,i,a-k+[k+1))=(i,i,[a+1))
\end{equation*}
і
\begin{equation*}
  \mathfrak{h}^k(i,i,[a))\cdot\mathfrak{h}^k(j,j,[k+1))=(i,i,[a)),
\end{equation*}
а це суперечить тому, що відображення $\mathfrak{h}^k\colon \boldsymbol{B}_{\omega}^{\mathscr{F}_{\leqslant k+1}}\to \boldsymbol{B}_{\omega}^{\mathscr{F}_{\leqslant k}}$ є гомоморфізмом. З отриманого протиріччя випливає, що умова $\mathfrak{h}^k(0,0,[k+1))=(0,0,[k))$ не може виконуватися.

\medskip

Припустимо, що виконується випадок \eqref{theorem-3.6-proof(2)}: $\mathfrak{h}^k(0,0,[k+1))=(1,1,[k))$. З твердження \ref{proposition-2.6} випливає, що $\mathfrak{h}^k(i,j,[k+1))\in \boldsymbol{B}_{\omega}^{\{[k)\}}$ для всіх $i,j\in\omega$, оскільки за твердженням~3 з \cite{Gutik-Mykhalenych-2020} напівгрупа $\boldsymbol{B}_{\omega}^{\{[k+1)\}}$ ізоморфна біциклічній напівгрупі. Також з тверджень \ref{proposition-2.5} і \ref{proposition-2.9} випливає, що $\mathfrak{h}^k(1,1,[k+1))\preccurlyeq(2,2,[k))$.

За твердженням~1.4.21(2) з \cite{Lawson-1998} маємо, що $\mathfrak{h}^k(1,1,[k+1))$~--- ідемпотент напівгрупи $\boldsymbol{B}_{\omega}^{\mathscr{F}}$, а отже, $\mathfrak{h}^k(1,1,[k+1))=(i,i,[k))$ для деякого натурального числа $i\geqslant 3$. Позаяк
\begin{equation*}
  (0,0,[k+1))=(0,1,[k+1))\cdot(1,0,[k+1)),
\end{equation*}
то за твердженням~1.4.21(1) \cite{Lawson-1998} отримуємо, що
\begin{align*}
  (1,1,[k))&=\mathfrak{h}^k(0,0,[k+1))= \\
  &=\mathfrak{h}^k((0,1,[k+1))\cdot(1,0,[k+1)))=\\
  &=\mathfrak{h}^k(0,1,[k+1))\cdot\mathfrak{h}^k(1,0,[k+1))=\\
  &=\mathfrak{h}^k(0,1,[k+1))\cdot\mathfrak{h}^k((0,1,[k+1))^{-1})=\\
  &=\mathfrak{h}^k(0,1,[k+1))\cdot(\mathfrak{h}^k(0,1,[k+1)))^{-1},
\end{align*}
і, прийнявши $\mathfrak{h}^k(0,1,[k+1))=(m,n,[k))$, з леми 4 \cite{Gutik-Mykhalenych-2020} випливає, що
\begin{equation*}
  (m,n,[k))\cdot(m,n,[k))^{-1}=(m,n,[k))\cdot(n,m,[k))=(m,m,[k))=(1,1,[k)),
\end{equation*}
звідки отримуємо, що $\mathfrak{h}^k(0,1,[k+1))=(1,n,[k))$, для деякого натурального числа $n\geqslant 2$. Справді, $n\neq 1$, оскільки в цьому випадку за теоремою~\ref{theorem-2.7} гомоморфізм $\mathfrak{h}^k\colon \boldsymbol{B}_{\omega}^{\mathscr{F}_{\leqslant k+1}}\to \boldsymbol{B}_{\omega}^{\mathscr{F}_{\leqslant k}}$ був би груповим. Також, з нерівності $(1,1,[k+1))\preccurlyeq(0,0,[k+1))$ в $E(\boldsymbol{B}_{\omega}^{\mathscr{F}_{\leqslant k+1}})$ випливає, що не існує ідемпотента вигляду $(p,p,[k+1))$, де $p\geqslant 2$, такого, що $\mathfrak{h}^k(p,p,[k+1))=(0,0,[k)))$, а отже за твердженням~1.4.21(3) з \cite{Lawson-1998} не існує елемента $(s,t,[k+1))$ напівгрупи $\boldsymbol{B}_{\omega}^{\mathscr{F}_{\leqslant k+1}}$ такого, що $\mathfrak{h}^k(s,t,[k+1))=(0,0,[k))$.

Аналогічно,  з рівностей
\begin{align*}
  (i,i,[k))&= \mathfrak{h}^k(1,1,[k+1))=\\
  &=\mathfrak{h}^k((1,0,[k+1))\cdot(0,1,[k+1)))=\\
  &=\mathfrak{h}^k(1,0,[k+1)))\cdot\mathfrak{h}^k(0,1,[k+1))=\\
  &=\mathfrak{h}^k((0,1,[k+1))^{-1})\cdot\mathfrak{h}^k(0,1,[k+1))=\\
  &=(\mathfrak{h}^k(0,1,[k+1)))^{-1}\cdot\mathfrak{h}^k(0,1,[k+1))=\\
  &=(1,n,[k))^{-1}\cdot(1,n,[k))=\\
  &=(n,1,[k))\cdot(1,n,[k))=\\
  &=(n,n,[k))
\end{align*}
отримуємо, що $n=i\geqslant 2$.

Припустимо, що $i\geqslant 3$. Тоді з
\begin{align*}
  (3,3,[k))&\preccurlyeq(2,2,[k+1))=\\
  &=(2,0,[k+1))\cdot(0,2,[k+1))=\\
  &=(1,0,[k+1))^2\cdot(0,1,[k+1))^2
\end{align*}
випливає, що
\begin{align*}
  (3,3,[k))&=\mathfrak{h}^k(3,3,[k))\preccurlyeq\\
  &\preccurlyeq\mathfrak{h}^k(2,2,[k+1))=\\
  &=\mathfrak{h}^k((1,0,[k+1))^2\cdot(0,1,[k+1))^2)=\\
  &=(\mathfrak{h}^k(1,0,[k+1)))^2\cdot(\mathfrak{h}^k(0,1,[k+1)))^2=\\
  &=(i,1,[k))^2\cdot(1,i,[k))^2=\\
  &=(2i-1,1,[k))\cdot(1,2i-1,[k))=\\
  &=(2i-1,2i-1,[k)),
\end{align*}
однак це суперечить твердженням \ref{proposition-2.5} і \ref{proposition-2.9}, оскільки $(3,3,[k))\not\preccurlyeq(2i-1,2i-1,[k))$ у випадку $i\geqslant 3$. Отож, $i=2$, а отже отримуємо, що $\mathfrak{h}^k(1,0,[k+1))=(1,2,[k))$ i $\mathfrak{h}^k(0,1,[k+1))=(2,1,[k))$. Позаяк за твердженням~3 з \cite{Gutik-Mykhalenych-2020} напівгрупа $\boldsymbol{B}_{\omega}^{\{[k+1)\}}$ ізоморфна біциклічній напівгрупі, то з того, що звуження гомоморфізму $\mathfrak{h}^k\colon \boldsymbol{B}_{\omega}^{\mathscr{F}_{\leqslant k+1}}\to \boldsymbol{B}_{\omega}^{\mathscr{F}_{\leqslant k}}$ на піднапівгрупу $\boldsymbol{B}_{\omega}^{\{[k+1)\}}$ є ін'єктивним відоб\-ра\-жен\-ням, з означення напівгрупової операції в $\boldsymbol{B}_{\omega}^{\mathscr{F}}$ і з попередніх міркувань випливає, що 
\begin{equation*}
\mathfrak{h}^k(p,q,[k+1))=(p+1,q+1,[k)),
\end{equation*}
для довільних $p,q\in\omega$.

Припустимо, що $i>j$. Тоді для $a=0,1,\ldots,k$ маємо
\begin{align*}
  \mathfrak{h}^k((i,i,[a))\cdot(j,j,[k+1)))&=\mathfrak{h}^k(i,i,[a)\cap(j-i+[k+1)))= \\
  &=
  \left\{
    \begin{array}{ll}
      \mathfrak{h}^k(i,i,j-i+[k+1)), & \hbox{якщо~} 0\leqslant a<j-i+k+1;\\
      \mathfrak{h}^k(i,i,[a)),       & \hbox{якщо~} a\geqslant j-i+k+1
    \end{array}
  \right.
  =\\
  &=
  \left\{
    \begin{array}{ll}
      (i+1,i+1,j-i+[k+1)), & \hbox{якщо~} 0\leqslant a<j-i+k+1;\\
      (i+1,i+1,[a)),       & \hbox{якщо~} a\geqslant j-i+k+1
    \end{array}
  \right.
\end{align*}
i
\begin{align*}
  \mathfrak{h}^k(i,i,[a))\cdot\mathfrak{h}^k(j,j,[k+1))&=(i+1,i+1,[a))\cdot(j+1,j+1,[k))= \\
  &=(i+1,i+1,[a)\cap(j-i+[k)))=\\
  &=
  \left\{
    \begin{array}{ll}
      (i+1,i+1,j-i+[k)), & \hbox{якщо~} 0\leqslant a<j-i+k;\\
      (i+1,i+1,[a)),       & \hbox{якщо~} a\geqslant j-i+k.
    \end{array}
  \right.
\end{align*}
Отож, якщо $a=j-i+k$ (а такий випадок завжди можливий зокрема коли $a=-1+k$ тобто $i=j+1$), то отримуємо, що
\begin{equation*}
  \mathfrak{h}^k((i,i,[a))\cdot(j,j,[k+1)))=(i+1,i+1,j-i+[k+1))=(i+1,i+1,a-k+[k+1))=(i+1,i+1,[a+1))
\end{equation*}
і
\begin{equation*}
  \mathfrak{h}^k(i,i,[a))\cdot\mathfrak{h}^k(j,j,[k+1))=(i+1,i+1,[a)),
\end{equation*}
а це суперечить тому, що відображення $\mathfrak{h}^k\colon \boldsymbol{B}_{\omega}^{\mathscr{F}_{\leqslant k+1}}\to \boldsymbol{B}_{\omega}^{\mathscr{F}_{\leqslant k}}$ є гомоморфізмом. З отриманого протиріччя випливає, що умова $\mathfrak{h}^k(0,0,[k+1))=(1,1,[k))$ не може виконуватися.

\medskip

Припустимо, що виконується випадок \eqref{theorem-3.6-proof(3)}: $\mathfrak{h}^k(0,0,[k+1))=(1,1,[k-1))$. З твердження \ref{proposition-2.6} випливає, що $\mathfrak{h}^k(i,j,[k+1))\in \boldsymbol{B}_{\omega}^{\{[k-1)\}}$ для всіх $i,j\in\omega$, оскільки за твердженням~3 з \cite{Gutik-Mykhalenych-2020} напівгрупа $\boldsymbol{B}_{\omega}^{\{[k+1)\}}$ ізоморфна біциклічній напівгрупі. Також з тверджень \ref{proposition-2.5} і \ref{proposition-2.9} випливає, що виконується нерівність $\mathfrak{h}^k(1,1,[k+1))\preccurlyeq(2,2,[k-1))$.

За твердженням~1.4.21(2) з \cite{Lawson-1998} маємо, що $\mathfrak{h}^k(1,1,[k+1))$~--- ідемпотент напівгрупи $\boldsymbol{B}_{\omega}^{\mathscr{F}}$, а отже, $\mathfrak{h}^k(1,1,[k+1))=(i,i,[k-1))$ для деякого натурального числа $i\geqslant 2$. Позаяк
\begin{equation*}
  (0,0,[k+1))=(0,1,[k+1))\cdot(1,0,[k+1)),
\end{equation*}
то за твердженням~1.4.21(1) \cite{Lawson-1998} маємо, що
\begin{align*}
  (1,1,[k-1))&=\mathfrak{h}^k(0,0,[k+1))= \\
  &=\mathfrak{h}^k((0,1,[k+1))\cdot(1,0,[k+1)))=\\
  &=\mathfrak{h}^k(0,1,[k+1))\cdot\mathfrak{h}^k(1,0,[k+1))=\\
  &=\mathfrak{h}^k(0,1,[k+1))\cdot\mathfrak{h}^k((0,1,[k+1))^{-1})=\\
  &=\mathfrak{h}^k(0,1,[k+1))\cdot(\mathfrak{h}^k(0,1,[k+1)))^{-1},
\end{align*}
і, прийнявши $\mathfrak{h}^k(0,1,[k+1))=(m,n,[k-1))$, з леми 4 \cite{Gutik-Mykhalenych-2020} випливає, що
\begin{equation*}
  (m,n,[k-1))\cdot(m,n,[k-1))^{-1}=(m,n,[k-1))\cdot(n,m,[k-1))=(m,m,[k))=(1,1,[k-1)),
\end{equation*}
а отже, $\mathfrak{h}^k(0,1,[k+1))=(1,n,[k-1))$, для деякого натурального числа $n\geqslant 2$. Справді, $n\neq 1$, оскільки в цьому випадку за теоремою~\ref{theorem-2.7} гомоморфізм $\mathfrak{h}^k\colon \boldsymbol{B}_{\omega}^{\mathscr{F}_{\leqslant k+1}}\to \boldsymbol{B}_{\omega}^{\mathscr{F}_{\leqslant k}}$ був би груповим. Також, з нерівності $(1,1,[k+1))\preccurlyeq(0,0,[k+1))$ в $E(\boldsymbol{B}_{\omega}^{\mathscr{F}_{\leqslant k+1}})$ випливає, що не існує ідемпотента вигляду $(p,p,[k+1))$, де $p\geqslant 2$, такого, що $\mathfrak{h}^k(p,p,[k+1))=(0,0,[k-1)))$, а отже за твердженням~1.4.21(3) з \cite{Lawson-1998} не існує елемента $(s,t,[k+1))$ напівгрупи $\boldsymbol{B}_{\omega}^{\mathscr{F}_{\leqslant k+1}}$ такого, що $\mathfrak{h}^k(s,t,[k+1))=(0,0,[k-1))$.

Аналогічно,  з рівностей
\begin{align*}
  (i,i,[k-1))&= \mathfrak{h}^k(1,1,[k+1))=\\
  &=\mathfrak{h}^k((1,0,[k+1))\cdot(0,1,[k+1)))=\\
  &=\mathfrak{h}^k(1,0,[k+1)))\cdot\mathfrak{h}^k(0,1,[k+1))=\\
  &=\mathfrak{h}^k((0,1,[k+1))^{-1})\cdot\mathfrak{h}^k(0,1,[k+1))=\\
  &=(\mathfrak{h}^k(0,1,[k+1)))^{-1}\cdot\mathfrak{h}^k(0,1,[k+1))=\\
  &=(1,n,[k-1))^{-1}\cdot(1,n,[k-1))=\\
  &=(n,1,[k-1))\cdot(1,n,[k-1))=\\
  &=(n,n,[k-1))
\end{align*}
отримуємо, що $n=i\geqslant 2$.

Припустимо, що $i\geqslant 3$. Тоді з
\begin{align*}
  (6,6,[k-1))&\preccurlyeq(4,4,[k+1))=\\
  &=(4,0,[k+1))\cdot(0,4,[k+1))=\\
  &=(1,0,[k+1))^4\cdot(0,1,[k+1))^4
\end{align*}
випливає, що
\begin{align*}
  (6,6,[k-1))&=\mathfrak{h}^k(6,6,[k-1))\preccurlyeq\\
  &\preccurlyeq\mathfrak{h}^k(4,4,[k+1))=\\
  &=\mathfrak{h}^k((1,0,[k+1))^4\cdot(0,1,[k+1))^4)=\\
  &=(\mathfrak{h}^k(1,0,[k+1)))^4\cdot(\mathfrak{h}^k(0,1,[k+1)))^4=\\
  &=(i,1,[k-1))^4\cdot(1,i,[k-1))^4=\\
  &=(4i-3,1,[k-1))\cdot(1,4i-3,[k-1))=\\
  &=(4i-3,4i-3,[k-1)),
\end{align*}
однак це суперечить твердженням \ref{proposition-2.5} і \ref{proposition-2.9}, оскільки $(6,6,[k-1))\not\preccurlyeq(4i-3,4i-3,[k-1))$ у випадку $i\geqslant 3$. Отож, $i=2$, а отже отримуємо, що $\mathfrak{h}^k(1,0,[k+1))=(1,2,[k-1))$ i $\mathfrak{h}^k(0,1,[k+1))=(2,1,[k))$. Позаяк за твердженням~3 з \cite{Gutik-Mykhalenych-2020} напівгрупа $\boldsymbol{B}_{\omega}^{\{[k+1)\}}$ ізоморфна біциклічній напівгрупі, то з того, що звуження гомоморфізму $\mathfrak{h}^k\colon \boldsymbol{B}_{\omega}^{\mathscr{F}_{\leqslant k+1}}\to \boldsymbol{B}_{\omega}^{\mathscr{F}_{\leqslant k}}$ на піднапівгрупу $\boldsymbol{B}_{\omega}^{\{[k+1)\}}$ є ін'єктивним відображенням, з означення напівгрупової операції в $\boldsymbol{B}_{\omega}^{\mathscr{F}}$ і з попередніх міркувань випливає, що 
\begin{equation*}
\mathfrak{h}^k(p,q,[k+1))=(p+1,q+1,[k-1)),
\end{equation*}
для довільних $p,q\in\omega$.

Припустимо, що $i>j$. Тоді для $a=0,1,\ldots,k$ маємо
\begin{align*}
  \mathfrak{h}^k((i,i,[a))\cdot(j,j,[k+1)))&=\mathfrak{h}^k(i,i,[a)\cap(j-i+[k+1)))= \\
  &=
  \left\{
    \begin{array}{ll}
      \mathfrak{h}^k(i,i,j-i+[k+1)), & \hbox{якщо~} 0\leqslant a<j-i+k+1;\\
      \mathfrak{h}^k(i,i,[a)),       & \hbox{якщо~} a\geqslant j-i+k+1
    \end{array}
  \right.
  =\\
  &=
  \left\{
    \begin{array}{ll}
      (i+1,i+1,j-i+[k+1)), & \hbox{якщо~} 0\leqslant a<j-i+k+1;\\
      (i+1,i+1,[a)),       & \hbox{якщо~} a\geqslant j-i+k+1
    \end{array}
  \right.
\end{align*}
i
\begin{align*}
  \mathfrak{h}^k(i,i,[a))\cdot\mathfrak{h}^k(j,j,[k+1))&=(i+1,i+1,[a))\cdot(j+1,j+1,[k-1))= \\
  &=(i+1,i+1,[a)\cap(j-i+[k-1)))=\\
  &=
  \left\{
    \begin{array}{ll}
      (i+1,i+1,j-i+[k-1)), & \hbox{якщо~} 0\leqslant a<j-i+k-1;\\
      (i+1,i+1,[a)),       & \hbox{якщо~} a\geqslant j-i+k-1.
    \end{array}
  \right.
\end{align*}
Отож, якщо $a=j-i+k$ (а такий випадок завжди можливий зокрема коли $a=-1+k$ тобто $i=j+1$), то отримуємо, що
\begin{equation*}
  \mathfrak{h}^k((i,i,[a))\cdot(j,j,[k+1)))=(i+1,i+1,j-i+[k+1))=(i+1,i+1,a-k+[k+1))=(i+1,i+1,[a+1))
\end{equation*}
і
\begin{equation*}
  \mathfrak{h}^k(i,i,[a))\cdot\mathfrak{h}^k(j,j,[k+1))=(i+1,i+1,[a)),
\end{equation*}
а це суперечить тому, що відображення $\mathfrak{h}^k\colon \boldsymbol{B}_{\omega}^{\mathscr{F}_{\leqslant k+1}}\to \boldsymbol{B}_{\omega}^{\mathscr{F}_{\leqslant k}}$ є гомоморфізмом. З отриманого протиріччя випливає, що умова $\mathfrak{h}^k(0,0,[k+1))=(1,1,[k-1))$ не може виконуватися.

Таким чином, жоден з випадків \eqref{theorem-3.6-proof(1)}, \eqref{theorem-3.6-proof(2)}, чи \eqref{theorem-3.6-proof(3)} не виконується, а отже не існує гомоморфної ретракції $\mathfrak{h}^k\colon \boldsymbol{B}_{\omega}^{\mathscr{F}_{\leqslant k+1}}\to \boldsymbol{B}_{\omega}^{\mathscr{F}_{\leqslant k}}$, звідки і випливають твердження теореми.
\end{proof}

З тверджень \ref{proposition-3.2}, \ref{proposition-3.5} і теореми \ref{theorem-3.6} випливає теорема \ref{theorem-3.7}, яка описує всі нетривіальні гомоморфні ретракти напівгрупи $\boldsymbol{B}_{\omega}^{\mathscr{F}}$.

\begin{theorem}\label{theorem-3.7}
Нехай $\mathscr{F}$ --- довільна ${\omega}$-замкнена сім'я індуктивних непорожніх підмножин у $\omega$. Якщо сім'я $\mathscr{F}$ --- нескінченна, то для довільного натурального числа $i$ та для довільного $j\in\omega$ напівгрупи $\boldsymbol{B}_{\omega}^{\mathscr{F}_{\geqslant i}}$ i $\boldsymbol{B}_{\omega}^{\{[j)\}}$ і лише вони, є нетривіальними гомоморфними ретрак\-та\-ми напівгрупи $\boldsymbol{B}_{\omega}^{\mathscr{F}}$.
Якщо ж сім'я $\mathscr{F}$ --- скінченна та $\displaystyle\bigcap\mathscr{F}=[k)$ для деякого натурального числа $k\geqslant 2$, то для довільного натурального числа $i\leqslant k$ та для довільного $j=0,1,\ldots,k$ напівгрупи $\boldsymbol{B}_{\omega}^{\mathscr{F}_{\geqslant i}}$ i $\boldsymbol{B}_{\omega}^{\{[j)\}}$  і лише вони, є нетривіальними гомоморфними ретрактами напівгрупи $\boldsymbol{B}_{\omega}^{\mathscr{F}}$, а у випадку $\displaystyle\bigcap\mathscr{F}=[1)$ напівгрупи $\boldsymbol{B}_{\omega}^{\{[0)\}}$ і $\boldsymbol{B}_{\omega}^{\{[1)\}}$  і лише вони, є нетривіальними гомоморфними ретрактами напівгрупи $\boldsymbol{B}_{\omega}^{\mathscr{F}}$.
\end{theorem}

\begin{remark}\label{remark-3.8}
За твердженням \ref{proposition-2.1} для довільної ${\omega}$-замкненої сім'ї $\mathscr{F}$ індуктивних непорожніх підмножин у $\omega$ напівгрупа $\boldsymbol{B}_{\omega}^{\mathscr{F}}$ ізоморфна напівгрупі $\boldsymbol{B}_{\omega}^{\mathscr{F}_0}$, де  $\mathscr{F}_0$ --- ${\omega}$-замкнена сім'я індуктивних непорожніх підмножин у $\omega$ та $\omega\in \mathscr{F}_0$. Отож, теорема \ref{theorem-3.7} описує всі нетривіальні гомоморфні ретракти напівгрупи $\boldsymbol{B}_{\omega}^{\mathscr{F}}$ за модулем ізоморфізму напівгруп $\boldsymbol{B}_{\omega}^{\mathscr{F}}$ і $\boldsymbol{B}_{\omega}^{\mathscr{F}_0}$, який описаний у доведенні твердження \ref{proposition-2.1}.
\end{remark}

Завершимо цю працю твердженнями, які описують ізоморфізм між напівгрупами $\boldsymbol{B}_{\omega}^{\mathscr{F}_1}$ і $\boldsymbol{B}_{\omega}^{\mathscr{F}_2}$ у випадку, коли $\mathscr{F}_1$ і $\mathscr{F}_2$ --- ${\omega}$-замкнені сім'ї індуктивних непорожніх підмножин у $\omega$.

\begin{theorem}\label{theorem-3.9}
Нехай $\mathscr{F}_1$ і $\mathscr{F}_2$ --- ${\omega}$-замкнені сім'ї індуктивних непорожніх підмножин у $\omega$. Тоді такі умови еквівалентні:
\begin{enumerate}
  \item\label{theorem-3.9(1)} напівгрупи $\boldsymbol{B}_{\omega}^{\mathscr{F}_1}$ і $\boldsymbol{B}_{\omega}^{\mathscr{F}_2}$ ізоморфні;
  \item\label{theorem-3.9(2)} сім'ї $\mathscr{F}_1$ і $\mathscr{F}_2$ рівнопотужні;
  \item\label{theorem-3.9(3)} існує ціле число $n$ таке, що $\mathscr{F}_1=\left\{n+F\colon F\in\mathscr{F}_2\right\}$.
\end{enumerate}
\end{theorem}

\begin{proof}
Імплікація \eqref{theorem-3.9(3)}$\Rightarrow$\eqref{theorem-3.9(2)} очевидна, а імплікація \eqref{theorem-3.9(3)}$\Rightarrow$\eqref{theorem-3.9(1)} випливає з твердження~\ref{proposition-2.1}.

\smallskip

\eqref{theorem-3.9(2)}$\Rightarrow$\eqref{theorem-3.9(3)}. Припустимо, що сім'ї $\mathscr{F}_1$ і $\mathscr{F}_2$ рівнопотужні. Приймемо $n_1=\min\displaystyle\bigcup\mathscr{F}_1$ i $n_2=\min\displaystyle\bigcup\mathscr{F}_2$. Тоді $[n_1)\in \mathscr{F}_1$ i $[n_2)\in \mathscr{F}_2$. Якщо сім'ї $\mathscr{F}_1$ і $\mathscr{F}_2$ --- скінченні, то існують максимальні цілі числа $n_1^0$ i $n_2^0$ такі, що $[n_1^0)\in \mathscr{F}_1$ i $[n_2^0)\in \mathscr{F}_2$, але $[n_1^0+1)\notin \mathscr{F}_1$ i $[n_2^0+1)\notin \mathscr{F}_2$. З леми~\ref{lemma-2.3} випливає, що  $[i)\in \mathscr{F}_1$ i $[j)\in \mathscr{F}_2$ для довільних цілих чисел $i=n_1,\ldots,n_1^0$ i $j=n_2,\ldots,n_2^0$. Тоді з рівнопотужності сімей $\mathscr{F}_1$ і $\mathscr{F}_2$ отримуємо, що
\begin{equation*}
  \left|\mathscr{F}_1\right|=n_1^0-n_1+1=n_2^0-n_2+1=\left|\mathscr{F}_2\right|,
\end{equation*}
і прийнявши $n=n_1-n_2=n_1^0-n_2^0$, отримуємо, що $\mathscr{F}_1=\left\{n+F\colon F\in\mathscr{F}_2\right\}$.

Якщо сім'ї $\mathscr{F}_1$ і $\mathscr{F}_2$ --- нескінченні, то з леми~\ref{lemma-2.3} випливає, що  $[i)\in \mathscr{F}_1$ i $[j)\in \mathscr{F}_2$ для довільних цілих чисел $i\geqslant n_1$ i $j\geqslant n_2$, а отже $\mathscr{F}_1=\left\{n+F\colon F\in\mathscr{F}_2\right\}$ для $n=n_1-n_2$.

\smallskip

\eqref{theorem-3.9(1)}$\Rightarrow$\eqref{theorem-3.9(2)}. Припустимо, що відображення $\mathfrak{h}\colon \boldsymbol{B}_{\omega}^{\mathscr{F}_1}\rightarrow\boldsymbol{B}_{\omega}^{\mathscr{F}_2}$ є ізоморфізмом. Приймемо $n_1=\min\displaystyle\bigcup\mathscr{F}_1$ i $n_2=\min\displaystyle\bigcup\mathscr{F}_2$. За теоремою~4 з \cite{Gutik-Mykhalenych-2020} елементи $(0,0,[n_1))$ i $(0,0,[n_2))$ є одиницями напівгруп $\boldsymbol{B}_{\omega}^{\mathscr{F}_1}$ і $\boldsymbol{B}_{\omega}^{\mathscr{F}_2}$, відповідно. З твердження~\ref{proposition-2.6} випливає, що гомоморфний образ $\mathfrak{h}(\boldsymbol{B}_{\omega}^{\{[n_1)\}})$ міститься в піднапівгрупі $\boldsymbol{B}_{\omega}^{\{[n_2)\}}$ напівгрупи $\boldsymbol{B}_{\omega}^{\mathscr{F}_2}$. Аналогічно, гомоморфний образ $\mathfrak{h}^{-1}(\boldsymbol{B}_{\omega}^{\{[n_2)\}})$ стосовно обереного відображення $\mathfrak{h}^{-1}$ до ізо\-мор\-фіз\-му $\mathfrak{h}\colon \boldsymbol{B}_{\omega}^{\mathscr{F}_1}\rightarrow\boldsymbol{B}_{\omega}^{\mathscr{F}_2}$, міститься в піднапівгрупі $\boldsymbol{B}_{\omega}^{\{[n_1)\}}$ напівгрупи $\boldsymbol{B}_{\omega}^{\mathscr{F}_1}$. Також, оскільки композиції $\mathfrak{h}^{-1}\circ\mathfrak{h}$ i $\mathfrak{h}\circ\mathfrak{h}^{-1}$ є тотожними відображеннями напівгруп $\boldsymbol{B}_{\omega}^{\mathscr{F}_1}$ і $\boldsymbol{B}_{\omega}^{\mathscr{F}_2}$, відповідно, то їх звуження $\left(\mathfrak{h}^{-1}\circ\mathfrak{h}\right)\!|_{\boldsymbol{B}_{\omega}^{\{[n_1)\}}}$ i $\left(\mathfrak{h}\circ\mathfrak{h}^{-1}\right)\!|_{\boldsymbol{B}_{\omega}^{\{[n_2)\}}}$ є тотожними відображеннями напівгруп $\boldsymbol{B}_{\omega}^{\{[n_1)\}}$ i $\boldsymbol{B}_{\omega}^{\{[n_2)\}}$, відповідно, а отже, звуження $\mathfrak{h}|_{\boldsymbol{B}_{\omega}^{\{[n_1)\}}}\colon \boldsymbol{B}_{\omega}^{\{[n_1)\}}\rightarrow\boldsymbol{B}_{\omega}^{\{[n_2)\}}$ i $\mathfrak{h}^{-1}|_{\boldsymbol{B}_{\omega}^{\{[n_2)\}}}\colon \boldsymbol{B}_{\omega}^{\{[n_1)\}}\rightarrow\boldsymbol{B}_{\omega}^{\{[n_1)\}}$ ізоморфізму $\mathfrak{h}$ і до нього оберненого $\mathfrak{h}^{-1}$ на напівгрупи $\boldsymbol{B}_{\omega}^{\{[n_1)\}}$ і $\boldsymbol{B}_{\omega}^{\{[n_2)\}}$, відповідно, є ізоморфізмами цих напівгруп.

Очевидно, що при ізоморфізмі $h\colon S\to T$ напівгруп $S$ i $T$ твірні елементи напівгрупи $S$ відоб\-ражаються у твірні елементи напівгрупи $T$. Тоді з того, що $(0,1)$ i $(1,0)$~--- твірні елементи бі\-цик\-ліч\-ного моноїда $\boldsymbol{B}_{\omega}$ i $(0,1)\cdot(1,0)=(0,0)$~--- одиниця в $\boldsymbol{B}_{\omega}$ та $(1,0)\cdot(0,1)=(1,1)$~--- ідемпотент бі\-цик\-лічного моноїда, який відмінний від одиниці, випливає, що тотожне відображення біциклічного моноїда та лише воно є ізоморфізмом з $\boldsymbol{B}_{\omega}$ на $\boldsymbol{B}_{\omega}$. З аналогічних міркувань i теореми~4~\cite{Gutik-Mykhalenych-2020} випливає, що для довільного числа $n\in\omega$ єдиний ізоморфізм $\mathfrak{f}\colon \boldsymbol{B}_{\omega}\to \boldsymbol{B}_{\omega}^{\{[n)\}}$ визначається за формулою
\begin{equation*}
\mathfrak{f}(i,j)=(i,j,[n)), \qquad i,j\in\omega,
\end{equation*}
а отже, звуження $\mathfrak{h}|_{\boldsymbol{B}_{\omega}^{\{[n_1)\}}}\colon \boldsymbol{B}_{\omega}^{\{[n_1)\}}\rightarrow\boldsymbol{B}_{\omega}^{\{[n_2)\}}$ ізоморфізму $\mathfrak{h}\colon \boldsymbol{B}_{\omega}^{\mathscr{F}_1}\rightarrow\boldsymbol{B}_{\omega}^{\mathscr{F}_2}$ визначається за формулою
\begin{equation*}
\mathfrak{h}|_{\boldsymbol{B}_{\omega}^{\{[n_1)\}}}(i,j,[n_1))=(i,j,[n_2)), \qquad i,j\in\omega.
\end{equation*}

Позаяк відображення $\mathfrak{h}\colon \boldsymbol{B}_{\omega}^{\mathscr{F}_1}\rightarrow\boldsymbol{B}_{\omega}^{\mathscr{F}_2}$ є ізоморфізмом, то з вище сказаного випливає, що його звуження $\mathfrak{h}\colon \boldsymbol{B}_{\omega}^{{\mathscr{F}_1}_{\geqslant n_1+1}}\rightarrow\boldsymbol{B}_{\omega}^{{\mathscr{F}_2}_{\geqslant n_2+1}}$ є ізоморфізмом. Далі аналогічно, як викладено вище, ін\-дук\-цією доводимо, що для довільного натурального числа $k$ такого, що $[n_1+k)\in \mathscr{F}_1$ звуження $\mathfrak{h}|_{\boldsymbol{B}_{\omega}^{\{[n_1+k)\}}}\colon \boldsymbol{B}_{\omega}^{\{[n_1+k)\}}\rightarrow\boldsymbol{B}_{\omega}^{\{[n_2+k)\}}$ ізоморфізму $\mathfrak{h}\colon \boldsymbol{B}_{\omega}^{\mathscr{F}_1}\rightarrow\boldsymbol{B}_{\omega}^{\mathscr{F}_2}$ визначається за формулою
\begin{equation*}
\mathfrak{h}|_{\boldsymbol{B}_{\omega}^{\{[n_1+k)\}}}(i,j,[n_1+k))=(i,j,[n_2+k)), \qquad i,j\in\omega.
\end{equation*}
Звідси випливає, що сім'ї $\mathscr{F}_1$ і $\mathscr{F}_2$ рівнопотужні.
\end{proof}

З доведення імплікації \eqref{theorem-3.9(1)}$\Rightarrow$\eqref{theorem-3.9(2)} теореми~\ref{theorem-3.9} випливає наслідок~\ref{corollary-3.10}, який описує структуру ізоморфізмів з напівгрупи $\boldsymbol{B}_{\omega}^{\mathscr{F}_1}$ у напівгрупу $\boldsymbol{B}_{\omega}^{\mathscr{F}_2}$.

\begin{corollary}\label{corollary-3.10}
Нехай $\mathscr{F}_1$ і $\mathscr{F}_2$ --- ${\omega}$-замкнені сім'ї індуктивних непорожніх підмножин у $\omega$, $n_1=\min\displaystyle\bigcup\mathscr{F}_1$ i $n_2=\min\displaystyle\bigcup\mathscr{F}_2$. Якщо напівгрупи $\boldsymbol{B}_{\omega}^{\mathscr{F}_1}$ і $\boldsymbol{B}_{\omega}^{\mathscr{F}_2}$ ізоморфні, то ізоморфізм $\mathfrak{h}\colon \boldsymbol{B}_{\omega}^{\mathscr{F}_1}\rightarrow\boldsymbol{B}_{\omega}^{\mathscr{F}_2}$ визначається за формулою
\begin{equation*}
\mathfrak{h}|_{\boldsymbol{B}_{\omega}^{\{[n_1+k)\}}}(i,j,[n_1+k))=(i,j,[n_2+k)), \qquad i,j\in\omega,
\end{equation*}
для кожного цілого числа $k$ такого, що $[n_1+k)\in \mathscr{F}_1$.
\end{corollary}

Нагадаємо \cite{Clifford-Preston-1961}, що \emph{автоморфізмом} напівгрупи $S$ називається довільний ізоморфізм з $S$ на $S$.

Оскільки автоморфізми тривіальної напівгрупи тривіальні, тобто є тотожними відображеннями, то з наслідку \ref{corollary-3.10} випливає

\begin{corollary}\label{corollary-3.11}
Для довільної ${\omega}$-замкненої сім'ї $\mathscr{F}$ індуктивних підмножин у $\omega$ кожен автоморфізм напівгрупи $\boldsymbol{B}_{\omega}^{\mathscr{F}}$ є тривіальним відображенням, а отже група авторморфізмів  напівгрупи $\boldsymbol{B}_{\omega}^{\mathscr{F}}$ тривіальна.
\end{corollary}

\end{document}